\documentclass[a4paper,10pt,
]{amsart}
\usepackage{amsmath,amssymb,amsthm}

\usepackage[alphabetic, abbrev]{amsrefs}
\usepackage{enumerate
}
\usepackage{hyperref}
\usepackage[all]{xy}
\newdir{ >}{{}*!/-5pt/@{>}} 
\SelectTips{cm}{} 

\usepackage[pagewise,switch]{lineno}

\numberwithin{equation}{section}
\newtheorem{theorem}[subsection]{Theorem}
\newtheorem{corollary}[subsection]{Corollary}
\newtheorem{lemma}[subsection]{Lemma}
\newtheorem{proposition}[subsection]{Proposition}
\theoremstyle{definition}
\newtheorem{definition}[subsection]{Definition}
\newtheorem{remark}[subsection]{Remark}
\newtheorem{example}[subsection]{Example}

\newcommand{\bC}{\mathbb{C}}

\newcommand{\bS}{\mathbb{S}}

\newcommand{\cA}{\mathcal{A}}

\newcommand{\cC}{\mathcal{C}}
\newcommand{\cD}{\mathcal{D}}
\newcommand{\cE}{\mathcal{E}}

\newcommand{\cI}{\mathcal{I}}

\newcommand{\cS}{\mathcal{S}}

\DeclareMathOperator*{\hocolim}{hocolim}
\DeclareMathOperator{\THH}{THH}

\DeclareMathOperator{\colim}{colim}

\DeclareMathOperator{\const}{const}

\newcommand{\ot}{\leftarrow}
\newcommand{\sm}{\wedge}
\newcommand{\iso}{\cong}

\newcommand{\bld}[1]{{\mathbf{#1}}}

\newcommand{\cy}{\mathrm{cy}}

\newcommand{\id}{{\mathrm{id}}}
\newcommand{\Sp}{\mathrm{Sp}}
\newcommand{\Spsym}[1]{{\mathrm{Sp}^{\Sigma}_{#1}}}

\DeclareMathOperator{\capitalGL}{GL}

\newcommand{\GLoneIof}[1]{\capitalGL^{\cI}_1\!\!{#1}}
\newcommand{\MGLoneIof}[1]{M\!\GLoneIof{#1}}

\DeclareMathOperator{\concat}{\sqcup}

\newcommand{\Top}{\mathrm{Top}}

\newcommand{\arxivlink}[1]{\href{http://arxiv.org/abs/#1}{\texttt{arXiv:#1}}}

\newcommand{\cof}{\mathrm{cof}}
\newcommand{\GL}{\mathrm{GL}}

\newcommand{\Map}{\mathrm{Map}}
\newcommand{\Tor}{\mathrm{Tor}}
\newcommand{\xr}{\xrightarrow}
\newcommand{\xl}{\xleftarrow}

\usepackage[pdftex]{graphics,color} 
\usepackage{eso-pic}
\usepackage{scrtime}
\usepackage{ifdraft}
\ifdraft{%
\AddToShipoutPicture{\put(30,30){\resizebox{!}{.4cm}%
   {{{\color[gray]{0.8}\texttt{Preliminary draft, compiled \the\year-\the\month-\the\day~at \thistime}}}}}}
}{}

\ifdraft{%
  \newcommand{\Bemerkung}[1]{{\marginpar{\hspace{0.2\marginparwidth}\rule{0.6\marginparwidth}{0.75mm}\hspace{0.2\marginparwidth}}\noindent\bfseries[#1]}}
}{
  \newcommand{\Bemerkung}[1]{}
}

\title[Generalized Thom spectra and their \texorpdfstring{$\THH$}{THH}]{Generalized Thom spectra and their \\  topological Hochschild homology}
\author{Samik Basu}
\address{Department of Mathematical and Computational Science, Indian Association for
the Cultivation of Science, Kolkata - 700032, India}
\email{samik.basu2@gmail.com}

\author{Steffen Sagave} \address{Radboud University Nijmegen, IMAPP, PO Box 9010, 6500 GL Nijmegen, \newline The Netherlands}  \email{s.sagave@math.ru.nl}

\author{Christian Schlichtkrull} \address{Department of Mathematics, University of Bergen, P.O. Box 7803, 5020 Bergen, \newline Norway} \email{christian.schlichtkrull@math.uib.no}

\date{\today}

\begin{document}
\begin{abstract}
We develop a theory of $R$-module Thom spectra for a commutative symmetric ring spectrum $R$ and we analyze their multiplicative properties. As an interesting source of examples, we show that $R$-algebra Thom spectra associated to the special unitary groups can be described in terms of quotient constructions on $R$. We apply the general theory to obtain a description of the $R$-based topological Hochschild homology associated to an $R$-algebra Thom spectrum. 
\end{abstract}
\maketitle
\ifdraft{\linenumbers}{} 

\section{Introduction}
In their most classical form, Thom spectra arise by forming Thom spaces of compatible families of vector bundles. The compatibility conditions amount to considering stable vector bundles and it is convenient to view the formation of such Thom spectra as a functor defined on the category of spaces over the classifying space $BO$ for stable vector bundles. This construction was extended by Mahowald and Lewis to spaces over $BF$, the classifying space for stable spherical fibrations. The paper by Lewis~\cite{LMS}*{IX} gives a comprehensive account of such ``classical'' Thom spectra with special emphasis on their multiplicative properties. 

In order to appreciate the relation to the generalized Thom spectra referred to in the title of the paper, one must first realize that $BF$ may be interpreted as the classifying space for the units of the sphere spectrum. It is by now well-known that every ``structured'' ring spectrum $R$ has an underlying grouplike monoid of units $\GL_1(R)$ which represents the functor that to a space $X$ associates the units 
in the ring of $R$-cohomology classes $R^0(X)$. The corresponding classifying space $B\GL_1(R)$ classifies spaces equipped with an action of $\GL_1(R)$. In the influential paper \cite{Ando-B-G-H-R_units-Thom} it was realized how to construct from a map  $f\colon X\to B\GL_1(R)$ an associated $R$-module Thom spectrum $M(f)$. Here $R$ is supposed to be a ring spectrum (or ``\mbox{$S$-algebra}'') in the symmetric monoidal category of spectra introduced in \cite{EKMM}, and $M(f)$ is an object in the associated module category. The main interest in \cite{Ando-B-G-H-R_units-Thom} is in developing the orientation theory for such \mbox{$R$-module} Thom spectra in a way that  generalizes the orientation theory introduced in
\cite{May-E-infty-ring-spaces}.

In the present paper we shall develop an analogous theory of generalized Thom spectra in the setting of diagram spectra. For definiteness, the main part of the paper is written in terms of symmetric spectra of topological spaces, but everything can be adapted to orthogonal spectra or symmetric spectra of simplicial sets, for example. (This is discussed in Remark~\ref{rem:alternative-setting}.) The overlap with the paper \cite{Ando-B-G-H-R_units-Thom} is quite small since our main interest is in the multiplicative properties of such generalized Thom spectra and in the associated topological Hochschild homology. In the following we describe the contents of the paper in more detail.
\subsection{\texorpdfstring{$\cI$}{I}-spaces and generalized Thom spectra}
 Let $R$ be a commutative symmetric ring spectrum which we assume to be semistable throughout the paper. (Semistability is a weak fibrancy condition on symmetric spectra, see 
Remark~\ref{rem:semistable}). Sometimes it is also necessary to impose a weak cofibrancy condition on $R$, but we suppress this in the introduction. Our generalized Thom spectrum functor takes values in the category of $R$-modules of symmetric spectra $\Sp^{\Sigma}_R$, equipped with the symmetric monoidal smash product $\wedge_R$. The most elegant way to express the multiplicative properties of such an $R$-module Thom spectrum functor is by realizing it as a lax symmetric monoidal functor. In order to facilitate this, we shall invoke the category of $\cI$-spaces, 
cf.\ \cites{Schlichtkrull_units,Schlichtkrull_Thom-symmetric,Blumberg-C-S_THH-Thom,Sagave-S_diagram}. Recall that $\cI$ denotes the skeleton category of finite sets and injections. We write $\Top^{\cI}$ for the category of $\cI$-spaces, that is, the category of functors from $\cI$ to the category of (compactly generated weak Hausdorff) topological spaces. A map of $\cI$-spaces $X\to Y$ is said to be an $\cI$-equivalence if the induced map of homotopy colimits $X_{h\cI}\to Y_{h\cI}$ is a weak homotopy equivalence. It is proved in \cite{Sagave-S_diagram} that the $\cI$-equivalences are the weak equivalences in a model structure on $\Top^{\cI}$ which makes it Quillen equivalent to the category of spaces $\Top$. The advantage of the category $\Top^{\cI}$ is that it has a symmetric monoidal convolution product in which every $E_{\infty}$ structure can be rectified to a strictly commutative monoid. In particular, it is shown in \cite{Schlichtkrull_units} and \cite{Sagave-S_diagram} that the units of $R$ can be conveniently modelled as a commutative $\cI$-space monoid $\GL_1^{\cI}(R)$. The latter has a classifying $\cI$-space which is again a commutative 
$\cI$-space monoid. Writing $BG$ for this classifying $\cI$-space, the fact that $BG$ is commutative implies that the over-category $\Top^{\cI}/BG$ inherits the structure of a symmetric monoidal category. The first version of our $R$-module Thom spectrum functor then takes the form of a lax symmetric monoidal functor 
\[
T^{\cI}\colon \Top^{\cI}/BG\to \Sp^{\Sigma}_R/M\GL_1^{\cI}(R)
\]
where $M\GL_1^{\cI}(R)$ denotes the $R$-module Thom spectrum associated to the terminal object in $\Top^{\cI}/BG$. We introduce this functor in Section~\ref{sec:gen-Thom-spectra} where we analyze its homotopical and multiplicative properties.  Here we also set up an appropriate Tor spectral sequence and we generalize Lewis and Mahowald's description of Thom spectra associated to suspensions. 

\subsection{Thom spectra associated to space level data}
We also want a version of our $R$-module Thom spectrum functor that takes ordinary space level data as input, and for this purpose it is convenient to use the homotopy colimit $BG_{h\cI}$ as a model of the classifying space for the units of $R$. The over-categories $\Top/BG_{h\cI}$ and $\Top^{\cI}/BG$ are related by a chain of Quillen equivalences and we define our space level Thom spectrum functor to be the composition
\[
T\colon \Top/BG_{h\cI}\xr{P_{BG}} \Top^{\cI}/BG\xr{T^{\cI}} \Sp^{\Sigma}_R/M\GL_1^{\cI}(R)
\] 
where $P_{BG}$ is an explicit lax monoidal functor that realizes the induced equivalence of homotopy categories. The functor so defined satisfies the conditions that one may require of a good point set level Thom spectrum functor: It takes weak homotopy equivalences over $BG_{h\cI}$ to stable equivalences of $R$-modules, and it preserves colimits, $h$-cofibrations, and tensors with unbased spaces; this is the content of Proposition~\ref{prop:properties-of-T}. The homotopy colimit $BG_{h\cI}$ has the structure of a topological monoid (associative but not strictly commutative) and the functor $T$ is lax monoidal with respect to the corresponding monoidal structure on $\Top/BG_{h\cI}$.
Since the units of $R$ usually cannot be realized as a strictly commutative monoid in $\Top$, we cannot make $T$ into a symmetric monoidal functor. What we have instead is a version of Lewis' results on preservation of operad actions. We show in Proposition~\ref{prop:T-operad-version} that if $\cD$ is an operad augmented over the Barratt-Eccles operad, then $T$ induces a functor on the corresponding categories of $\cD$-algebras
\[
T\colon \Top[\cD]/BG_{h\cI}\to \Sp^{\Sigma}_R[\cD]/M\GL_1^{\cI}(R).
\]
This shows in particular that if $f\colon M\to BG_{h\cI}$ is a map of topological monoids, then $T(f)$ is an $R$-algebra over 
$M\GL_1^{\cI}(R)$. In Appendix~\ref{app:loop-rectification} we set up a convenient passage from loop space data to topological monoids and in order to keep the notation simple we shall presently use the same notation $T(f)$ for the $R$-algebra Thom spectrum associated to a loop map $f$ (rather than a map of actual topological monoids). 

The analogy between the construction of $T$ and the Thom spectrum functor in~\cite{Ando-B-G-H-R_units-Thom} makes it  plausible that these two functors should be equivalent. This comparison will be addressed in the forthcoming paper~\cite{Sagave-S_Thom-comparison} where we also show that $T$ is equivalent to the $\infty$-categorical Thom spectrum functor introduced in~\cite{Ando-B-G-H-R_infinity-Thom}. 

\subsection{Quotient spectra as Thom spectra} As an interesting source of examples, we consider $R$-module Thom spectra associated to the special unitary groups $SU(n)$ in the case where $R$ is an even commutative symmetric ring spectrum (that is, the homotopy groups of $R$ are concentrated in even degrees). Such Thom spectra are analyzed in detail in 
Section~\ref{sec:Quotient-Thom} using a geometric approach. The result in the theorem below follows from the more elaborate statement in Theorem~\ref{thm:T(SU(n))-quotient}. We refer to Section~\ref{subsec:quotient-Thom} for a discussion of the $R$-module quotient spectrum $R/(u_1,\dots,u_n)$ associated to a sequence of homotopy classes $u_1,\dots,u_n$ in $\pi_*(R)$.

\begin{theorem}\label{thm:intro-Thom-quotient}
Suppose that $R$ is even and that $u_1,\dots,u_n$ is a sequence of homotopy classes with $u_i\in \pi_{2i}(R)$. Then there exists a loop map 
\[
f_{(u_1,\dots,u_n)}\colon SU(n+1)\to BG_{h\cI}
\] 
such that the homotopy type of the $R$-module underlying the associated $R$-algebra Thom spectrum $T(f_{(u_1,\dots,u_n)})$
 is determined by a stable equivalence
\[
T(f_{(u_1,\dots,u_n)})\simeq R/(u_1,\dots,u_n).
\]
\end{theorem}

This theorem can be applied in various ways. On the one hand it shows that the multiplicative structure of $SU(n+1)$ induces a multiplicative structure on the $R$-module $R/(u_1,\dots,u_{n})$. In the extreme case where all the classes $u_i$ are trivial this gives us the $R$-algebra $R\wedge SU(n+1)_+$ with multiplication inherited from $R$ and $SU(n+1)$. In the other extreme case,  when the classes $u_1,\dots,u_n$ form a regular sequence in $\pi_*(R)$, it follows that $T(f_{(u_1,\dots,u_n)})$ is a regular quotient of $R$ in the sense that there is an isomorphism
\[
\pi_*\big(T(f_{(u_1,\dots,u_n)})\big) \simeq \pi_*(R)/(u_1,\dots,u_n)
\] 
and $T(f_{(u_1,\dots,u_n)})$ is built from $R$ by iterated homotopy cofiber sequences as described in 
Section~\ref{sec:Quotient-Thom}. This should be compared to the work by Angeltveit \cite{Angeltveit-THH} who considers the special case where $u_1,\dots, u_n$ is a regular sequence and uses a different technique to show that in this case $R/(u_1,\dots,u_n)$ has an $A_{\infty}$ structure without any conditions on the degrees of the classes. In our setting we prove in Corollary~\ref{cor:T(SU(n))-quotient-2-periodic} that the above theorem holds without any restrictions on the even dimensional classes $u_i$ provided that $R$ is $2$-periodic as well as even. This leads to another application of the theorem: the verification that certain well-known spectra naturally arise as Thom spectra. 

\begin{example}\label{ex:2-periodic-Morava-Thom}
Let $E_n$ be the $2$-periodic Lubin-Tate spectrum with 
\[
\pi_*(E_n)=W(\mathbb F_{p^n})[[ u_1,\dots,u_{n-1}]][u,u^{-1}],\quad |u_i|=0,\  |u|=2.
\]
It is proved in \cite{Goerss-Hopkins_moduli} that $E_n$ has the structure of an $E_{\infty}$ ring spectrum, hence can be realized as a commutative symmetric ring spectrum. The $2$-periodic Morava \mbox{$K$-theory} spectrum $K_n$ is defined by
$
K_n=E_n/(p,u_1,\dots, u_{n-1}).
$
Thus, we have that $\pi_*(K_n)=\mathbb F_{p^n}[u,u^{-1}]$, and it follows from Corollary~\ref{cor:T(SU(n))-quotient-2-periodic} that there exists a loop map $f\colon SU(n+1)\to BG_{h\cI}$ (where $BG$ is the classifying $\cI$-space for $\GL_1^{\cI}(E_n)$), such that $K_n\simeq T(f)$.
\end{example}

A different but related construction appears in recent work by Hopkins
and Lurie ~\cite{Hopkins-L_Brauer} where the $\infty$-categorical
version of the $R$-algebra Thom spectrum functor
from~\cite{Ando-B-G-H-R_infinity-Thom} is used to study quotient spectra. More
specifically, they show that certain types of algebra spectra over
Lubin-Tate spectra arise as the Thom spectra associated with maps from
tori, and they use this to study Brauer groups of Lubin-Tate spectra.
\subsection{Topological Hochschild homology of Thom spectra} 
In Section~\ref{sec:THH-of-Thom} we use our results from the previous sections to analyse the $R$-based topological Hochschild homology $\THH^R(-)$ of $R$-algebra Thom spectra. This generalizes the analysis of topological Hochschild homology for ``classical'' Thom spectra in \cite{Blumberg-C-S_THH-Thom}. Let $f\colon X\to BG_{h\cI}$ be a loop map with delooping $Bf\colon BX\to B(BG_{h\cI})$. In this situation we shall introduce a certain map $L^{\eta}(Bf)\colon L(BX) \to BG_{h\cI}$ where $L(-)$ denotes the free loop space functor. The construction is given in Definition~\ref{def:L-eta-map} and generalizes that in \cite{Blumberg-C-S_THH-Thom}. We use the decoration $\eta$ to indicate a twist by the Hopf map arising from an incompatibility between the free loop space and the cyclic bar construction uncovered in \cite{Schlichtkrull_units}. The following theorem is derived from the statement in Theorem~\ref{thm:thh-thom} using the passage from loop space data to topological monoids detailed in Appendix~\ref{app:loop-rectification}.  
\begin{theorem}\label{thm:intro-THH-Thom}
Given a loop map $f\colon X\to BG_{h\cI}$ with associated $R$-algebra Thom spectrum $T(f)$, there is a stable equivalence  
$\THH^R(T(f))\simeq T(L^{\eta}(Bf))$. If $f$ is a $3$-fold loop map, then this simplifies to $\THH^R(T(f))\simeq T(f)\sm BX_+$.
\end{theorem}

It is possible to combine Theorems~\ref{thm:intro-Thom-quotient} and \ref{thm:intro-THH-Thom} in order to explicitly calculate the topological Hochschild homology of 
$R$-algebra quotient spectra. Such calculations are carried out in \cites{Basu_THH-of-K/p,Basu-S_THH-quotients} and provide a means for measuring the extent to which the induced multiplicative structure on the quotient spectrum $R/(u_1,\dots,u_n)$ depends on the choice of delooping of the map $f$ in Theorem~\ref{thm:intro-Thom-quotient}. As a sample calculation we offer the following example which we quote from \cite{Basu-S_THH-quotients}. Let again $E_n$ denote the Lubin-Tate spectrum and recall the structure of $\pi_*(E_n)$ described in Example~\ref{ex:2-periodic-Morava-Thom}. We write $\pi_*(E_n)/(p,u_1,\dots,u_{n-1})^{\infty}$ for the $\pi_*(E_n)$-module defined by
\[
\colim_{i,j_1,\dots,j_{n-1}}\pi_*(E_n)/(p^i, u_1^{j_1},\dots, u_{n-1}^{j-1}).
\]
\begin{example}[\cite{Basu-S_THH-quotients}] 
For each $k\geq 1$ such that $p\geq (n+1)(k+1)+1$, the \mbox{$2$-periodic} Morava $K$-theory spectrum $K_n$ admits a structure as an algebra over $E_n$ for which
\[
\pi_*\THH^{E_n}(K_n)\cong \bigoplus_{i=1}^k \pi_*(E_n)/(p,u_1,\dots,u_{n-1})^{\infty}
\]

This complements the calculations by Angeltveit \cite{Angeltveit-THH}.
\end{example}
As we will explain in detail in~\cite{Sagave-S_Thom-comparison}, the good
point set level properties of the present Thom spectrum functor also allow
one to express the $R$-based topological Andr\'e-{Q}uillen homology of 
$E_{\infty}$ $R$-algebra Thom spectra in terms of Thom spectra. This generalizes 
work by Basterra and Mandell~\cite{Bastera-M_Homology} for $\mathbb S$-algebra
Thom spectra. 
\subsection{Notation and conventions} Let $\Top$ denote the category of compactly generated weak Hausdorff topological spaces. We shall work with $\cI$-spaces and symmetric spectra in $\Top$ throughout the paper. The category of symmetric spectra is denoted by $\Sp^{\Sigma}$, and we write $\cA\Sp^{\Sigma}$ for the category of symmetric ring spectra and $\cC\Sp^{\Sigma}$ for the category of commutative symmetric ring spectra. We shall not have occasion to consider ring spectra in the weaker sense of being monoids in the stable homotopy category

\subsection{Organization} We start from scratch by reviewing basic material about \mbox{$\cI$-spaces} and their relation to symmetric spectra in Section~\ref{sec:I-spaces}. In Section~\ref{sec:gen-Thom-spectra} we set up the lax symmetric monoidal Thom spectrum functor $T^{\cI}$ taking $\cI$-space data as input, and in Section~\ref{sec:Thom-spectra-from-space-level-data} we use this to define the Thom spectrum functor $T$ taking ordinary space level data as input. The material on Thom spectra associated to the special unitary groups and the relation to quotient spectra is contained in Section~\ref{sec:Quotient-Thom}. We review the definition of topological Hochschild homology in terms of the cyclic bar construction in Section~\ref{sec:THH-of-Thom}, where we use this description in the proof of Theorem~\ref{thm:intro-THH-Thom}. In Section~\ref{sec:modules-classifying-spaces} we establish some useful facts about modules and classifying spaces for commutative $\cI$-space monoids. Appendix \ref{app:loop-rectification} is about the passage from loop space data to topological monoids.   

\section{\texorpdfstring{$\cI$}{I}-spaces and modules over \texorpdfstring{$\cI$}{I}-space monoids}\label{sec:I-spaces}
In this section we recall some basic facts about $\cI$-spaces and
symmetric spectra from~\cite{Sagave-S_diagram}*{Section~3}. We also formulate conditions on a commutative $\cI$-space monoid $G$ which ensure that the bar construction $BG$ classifies $G$-modules. 

\subsection{Review of \texorpdfstring{$\cI$}{I}-spaces}\label{subsec:I-spaces}
Let $\cI$ be the category with objects the finite sets of the form $\bld{m}=\{1,\dots,m\}$ for 
$m\geq 0$ (where $\bld 0$ denotes the empty set) and morphisms the injective maps. This is a symmetric strict monoidal category under ordered concatenation $-\concat-$ of ordered sets. Let $\Top^{\cI}$ be
the functor category of $\cI$-diagrams in $\Top$. The
monoidal structure on $\cI$ and the cartesian product of spaces induce
a convolution product $\boxtimes$ on $\Top^{\cI}$: For
$\cI$-spaces $X$ and $Y$, their product $X\boxtimes Y$ is the left Kan
extension of the $\cI\times\cI$-diagram $(\bld{k},\bld{l})\mapsto
X(\bld{k})\times Y(\bld{l})$ along $-\concat- \colon \cI\times\cI \to
\cI$. More explicitly, we have 
\[
(X \boxtimes Y)(\bld{m}) = \colim_{\bld{k}\concat\bld{l}\to\bld{m}}X(\bld{k})\times Y(\bld{l})
\]
where the colimit is taken over the comma category $(-\concat -\downarrow \cI)$.  
The terminal $\cI$-space $U^{\cI}=\cI(\bld{0},-)$ is the monoidal unit for $\boxtimes$. We use the term 
\emph{(commutative) $\cI$-space monoid} for a (commutative) monoid in the symmetric monoidal category $(\Top^{\cI},\boxtimes,U^{\cI})$. This amounts to the same thing as a lax (symmetric) monoidal functor from $\cI$ to $\Top$.

We now turn to the homotopy theory of $\cI$-spaces and write $X_{h\cI}$ (or $\hocolim_{\cI}X$)
for the Bousfield--Kan homotopy colimit of an $\cI$-space $X$ (see 
\cite{Bousfield-Kan_homotopy-limits}). 

\begin{definition}
A map of $\cI$-spaces $X\to Y$ is an \emph{$\cI$-equivalence} if $X_{h\cI}\to Y_{h\cI}$ is a weak homotopy equivalence of spaces. 
\end{definition}

The $\cI$-equivalences are the weak equivalences in several useful model structures on $\Top^{\cI}$ as discussed in 
\cite{Sagave-S_diagram}*{Section~3}. Since we shall be particularly interested in commutative $\cI$-space monoids and the associated module categories, it will be most convenient for our purposes to work with the so-called \emph{flat model structures}.  In order to describe the cofibrations in these model structures we need to review some basic equivariant homotopy theory for the symmetric groups 
$\Sigma_n$: The category of $\Sigma_n$-spaces admits a \emph{fine} (also known as the ``genuine'') model structure in which a map is a weak equivalence (or fibration) if and only if the induced map of $H$-fixed points is a weak homotopy equivalence (or fibration) for every subgroup $H$ in $\Sigma_n$. This is a cofibrantly generated model structure with generating cofibrations of the form $\Sigma_n/H\times S^{n-1}\to \Sigma_n/H\times D^n$ for $n\geq 0$ and $H$ any subgroup in $\Sigma_n$. 

Next recall that the \emph{$\bld n$th latching space} of an $\cI$-space $X$ is defined as the colimit $L_{\bld n}X=\colim_{\partial(\cI\downarrow \bld n)}X$, where  $\partial(\cI\downarrow \bld n)$ denotes the full subcategory of the comma category $(\cI\downarrow\bld n)$ with objects the non-isomorphisms. Here we view $X$ as a diagram over $\partial(\cI\downarrow \bld n)$ via the forgetful functor to $\cI$. The canonical action of $\Sigma_n$ on $\bld n$ induces a $\Sigma_n$-action on $L_{\bld n}X$. 

In the \emph{absolute flat model structure} on $\Top^{\cI}$, a map of $\cI$-spaces $X\to Y$ is 
\begin{itemize}
\item 
a weak equivalence if it is an $\cI$-equivalence,
\item
a cofibration if the induced \emph{latching map} 
$
X(\bld n)\cup_{L_{\bld n}(X)}L_{\bld n}(Y)\to Y(\bld n)
$
is a cofibration in the fine model structure on $\Sigma_n$-spaces for all $n\geq 0$, and 
\item
a fibration if it has the right lifting property with respect to cofibrations that are $\cI$-equivalences.    
\end{itemize}
The fibrations are described explicitly in \cite{Sagave-S_diagram}*{Section~6.11} and it follows from this description that if $X$ is a fibrant $\cI$-space in the absolute flat model structure, then any morphism $\bld m\to \bld n$ in $\cI$ induces a weak homotopy equivalence $X(\bld m)\to X(\bld n)$. The absolute flat model structure is a cofibrantly generated proper topological model structure that satisfies the pushout-product and the monoid axiom with respect to $\boxtimes$; 
see~\cite{Sagave-S_diagram}*{ Proposition~3.10}. We shall use the term \emph{flat $\cI$-space} for a cofibrant object in this model structure. If $X$ is a flat $\cI$-space, then the endofunctor $X\boxtimes -$ preserves $\cI$-equivalences by \cite{Sagave-S_diagram}*{Proposition~8.2}.

There is a variation of the absolute flat model structure on $\Top^{\cI}$ known as the \emph{positive flat model structure}, where the positivity condition is motivated by an insight of J. Smith, cf.\ ~\cite[\S 14]{MMSS}. The positive flat model structure again has the $\cI$-equivalences as its weak equivalences, but the conditions for a map 
$X\to Y$ to be a cofibration in the absolute flat model structure has been strengthened so that the latching map in degree zero (that is, the map $X(\bld 0)\to Y(\bld 0)$) is now supposed to be a homeomorphism.  
Consequently, the positive flat model structure has less cofibrations and more fibrations than the absolute flat model structure and in particular the condition for an $\cI$-space to be fibrant no longer implies that the initial map $\bld 0\to\bld n$ induces a weak homotopy equivalence $X(\bld 0)\to X(\bld n)$. The identity functor on $\Top^{\cI}$ is a left Quillen functor from the positive to the absolute flat model structure and defines a Quillen equivalence since these model structures have the same weak equivalences.  

\begin{remark}\label{rem:proj-remark}
Apart from the flat model structures there are also the so-called \emph{absolute} and \emph{positive projective model structures} on $\Top^{\cI}$, cf.\ \cite{Sagave-S_diagram}*{Section~3.1}. Let us temporarily write $\Top^{\cI}_{\mathrm {proj}}$  
for $\Top^{\cI}$ equipped with the (absolute or positive) projective model structure and $\Top^{\cI}_{\mathrm {flat}}$ for 
$\Top^{\cI}$ equipped with the corresponding (absolute or positive) flat model structure. Then there is a chain of Quillen equivalences 
\[
\xymatrix{
 \Top\ar@<-.5ex>[r]_-{\const_{\cI}} & \Top^{\cI}_{\textrm{proj}} \ar@<-.5ex>[l]_-{\mathrm{colim}_{\cI}} 
 \ar@<.5ex>[r]^-{\mathrm{id}}& \Top^{\cI}_{\mathrm{flat}} \ar@<.5ex>[l]^-{\mathrm{id}}
 }
\]
with respect to these model structures on $\Top^{\cI}$ and the usual Quillen model structure on $\Top$. 
Here the upper arrows indicate left Quillen functors. The first adjunction is induced by the colimit functor and the constant embedding, and the second adjunction is given by the identity functor. As a consequence, all these model structures have homotopy categories equivalent to the usual homotopy category of spaces.
\end{remark}

Now let $\cD$ be an operad in spaces as defined in \cite{May-geometry}, and let $\Top^{\cI}[\cD]$ be the category of 
$\cD$-algebras in $\Top^{\cI}$ with respect to the $\boxtimes$-product. The advantage of the positive flat model structure (as opposed to the absolute structure) is that it lifts to a positive flat model structure on $\Top^{\cI}[\cD]$ in the sense that a map of $\cD$-algebras is a weak equivalence or fibration if and only if the underlying map of $\cI$-spaces is so. This applies in particular to the commutativity operad $\cC$ (the terminal operad with $\cC_n=*$ for all $n$) and provides a positive flat model structure on the category of commutative $\cI$-space monoids $\cC\Top^{\cI} = \Top^{\cI}[\cC]$ that makes it Quillen equivalent to the category of $E_{\infty}$ spaces. 
In other words, the passage to $\cI$-spaces allows us to model $E_{\infty}$ spaces by strictly commutative monoids. In this paper we shall always use the term \emph{cofibrant commutative $\cI$-space monoid} to mean a cofibrant object in the positive flat model structure on $\cC\Top^{\cI}$. The flat model structures have the following convenient compatibility  property 
\cite{Sagave-S_diagram}*{Proposition~3.15}: If $G$ is a cofibrant commutative $\cI$-space monoid, then the underlying $\cI$-space of $G$ is flat.

The Bousfield-Kan homotopy colimit functor $(-)_{h\cI}\colon \Top^{\cI}\to \Top$ is a monoidal (but not symmetric monoidal) functor with monoidal product
\begin{equation}\label{eq:monoidal-structure-map-hI} X_{h\cI} \times Y_{h\cI} \xrightarrow{\iso} (X\times Y)_{h(\cI\times\cI)} \to (-\concat-)^*(X\boxtimes Y)_{h(\cI\times\cI)} \to (X\boxtimes Y)_{h\cI}
\end{equation}
 induced by the natural transformation $X(\bld{k})\times Y(\bld{l}) \to (X\boxtimes Y)(\bld{k}\concat\bld{l})$ resulting from the definition of $\boxtimes$ as a left Kan extension.  Hence an $\cI$-space monoid $M$ gives rise to a topological monoid $M_{h\cI}$. We say that $M$ is \emph{grouplike} if the monoid $\pi_0(M_{h\cI})$ is a group.

\subsection{Units of symmetric ring spectra} The category of $\cI$-spaces is related to the category of symmetric spectra $\Sp^{\Sigma}$ by an adjunction
\begin{equation}\label{eq:SI-OmegaI-adjunction} 
\bS^{\cI}\colon \Top^{\cI}\rightleftarrows \Spsym{} \colon \Omega^{\cI}
\end{equation} 
whose left adjoint $\bS^{\cI}$ is strong symmetric monoidal with respect to the $\boxtimes$-product on 
$\Top^{\cI}$ and the smash product on $\Sp^{\Sigma}$; see~\cite{Sagave-S_diagram}*{Section~3.17}. These functors are given explicitly at each level $n$ by $\bS^{\cI}[X]_{n} = S^n \sm X(\bld{n})_{+}$ and $\Omega^{\cI}(E)(\bld{n}) = \Omega^{n}(E_n)$.

 The \emph{absolute} and \emph{positive flat stable model structures} (also known as the  \emph{\mbox{S-model} structures}) on $\Sp^{\Sigma}$ discussed in \cite{Shipley_convenient} and \cite{Schwede_SymSp} are the analogues for symmetric spectra of the absolute and positive flat model structures on $\Top^{\cI}$ (and motivated their construction). The adjunction \eqref{eq:SI-OmegaI-adjunction} is a Quillen adjunction with respect to these model structures.
It is a pleasant fact that $\bS^{\cI}$ and $\Omega^{\cI}$ have better homotopy invariance properties than can be deduced from the general properties of a Quillen adjunction. 
\begin{lemma}\label{lem:hty-inv-of-SI-OmegaI} 
The left adjoint $\bS^{\cI}$ maps $\cI$-equivalences between arbitrary $\cI$-spaces to stable equivalences, and the right adjoint $\Omega^{\cI}$ maps $\pi_*$-isomorphisms between arbitrary symmetric spectra to $\cI$-equivalences.
\end{lemma}
\begin{proof} 
Since $\bS^{\cI}$ is a left Quillen functor and an acyclic fibration in the absolute flat model structure on $\Top^{\cI}$ is a level equivalence, the first statement follows from the fact that $\bS^{\cI}$ sends level equivalences of $\cI$-spaces to level equivalences of symmetric spectra.
For a $\pi_*$-isomorphism of symmetric spectra $E \to E'$  it follows from the definitions that $\Omega^{\cI}(E)\to \Omega^{\cI}(E')$ induces a weak homotopy equivalence when forming the homotopy colimit over the subcategory of $\cI$ given by the subset inclusions. 
By an argument originally due to J.~Smith, it is therefore also an \mbox{$\cI$-equivalence}; see~\cite{Shipley_THH}*{Proposition~2.2.9} or~\cite{Sagave-S_group-compl}*{Proposition~2.6}.
  \end{proof}
 
\begin{remark}\label{rem:semistable}
Recall from \cite{Shipley_THH} and \cite{Schwede_SymSp} that a symmetric spectrum is said to be \emph{semistable} if it is $\pi_*$-equivalent to a symmetric $\Omega$-spectrum. Since stable equivalences and $\pi_*$-isomorphisms agree for semistable symmetric spectra, it follows that $\Omega^{\cI}$ takes stable equivalences between semistable symmetric spectra to $\cI$-equivalences. Fortunately, most of the symmetric spectra that one encounters in practice are semistable. 
\end{remark}
   
Since $\bS^{\cI}$ is strong symmetric monoidal and $\Omega^{\cI}$ is lax symmetric monoidal, we have an induced adjunction
\begin{equation}\label{eq:CSI-CSpsym-adj} 
\bS^{\cI}\colon \cC\Top^{\cI}\rightleftarrows \cC\Spsym{} \colon \Omega^{\cI}
\end{equation} 
relating the category of commutative $\cI$-space monoids $\cC\Top^{\cI}$ to the category  of commutative symmetric ring spectra $\cC\Sp^{\Sigma}$. 
If $R$ is a semistable commutative symmetric ring spectrum, then $\Omega^{\cI}(R)$ is a commutative $\cI$-space monoid
model for the corresponding multiplicative $E_{\infty}$ space of $R$. 

\begin{definition}
Let $R$ be a semistable commutative symmetric ring spectrum.  The \emph{$\cI$-space units} $\GL^{\cI}_1(R)$ of $R$  is the sub commutative $\cI$-space monoid of $\Omega^{\cI}(R)$ given by the invertible path components in the sense that
  $\GL_1^{\cI}(R)(\bld{n})$ is the union of the path components in $\Omega^n(R_n)$ that represent units in the commutative ring $\pi_0(R)=\colim_n \pi_n(R_n)$.
 \end{definition}
It follows from the definition that  $\pi_0(\GL_1^{\cI}(R)_{h\cI})$ can be identified with the units in $\pi_0(R)\cong \pi_0(\Omega^{\cI}(R)_{h\cI})$ which shows that $\GL_1^{\cI}(R)$ is grouplike. We notice that the adjoint
of the inclusion $\GL_1^{\cI}(R)\to\Omega^{\cI}(R)$ provides a canonical map of commutative symmetric ring spectra
\begin{equation}\label{eq:SI-GLoneIR-to-R}
\bS^{\cI}[\GL_1^{\cI}(R)]\to R,
\end{equation}
analogous to the algebraic situation where a commutative ring receives a canonical map from the integral group ring of its units. 

\begin{remark}
If we want to consider the units of a commutative symmetric ring spectrum that is not semistable, we may apply the above construction to a suitable fibrant replacement.
\end{remark}

\subsection{The universal fibration \texorpdfstring{$EG\to BG$}{EG -> BG}}\label{subsec:class-spaces} 
Given an $\cI$-space monoid $G$, a right $G$-module $X$, and a left $G$-module $Y$, the simplicial two-sided bar construction $B_{\bullet}(X,G,Y)$ is the simplicial $\cI$-space $[n]\mapsto X\boxtimes G^{\boxtimes n}\boxtimes Y$ with simplicial structure maps defined as in \cite[Section 10]{May-geometry}. We write $B(X,G,Y)$ for the $\cI$-space defined by geometric realization of this simplicial object. When $G$ is commutative and $H$ and $H'$ are commutative $G$-algebras given by maps of commutative \mbox{$\cI$-space} monoids $G\to H$ and $G\to H'$, the two-sided bar construction $B(H,G,H')$ inherits the structure of a commutative $\cI$-space monoid, cf.\ \cite[Lemma 10.1]{May-geometry}. The monoidal unit $U^{\cI}$ is also a terminal object in $\cC\Top^{\cI}$ and we define the bar construction $BG$ to be the commutative $\cI$-space monoid $B(U^{\cI},G,U^{\cI})$.

\begin{definition}\label{def:EG-BG}
Given a commutative $\cI$-space monoid $G$, the map $G\to U^{\cI}$ induces a map 
$B(U^{\cI},G,G) \to BG$ of commutative $\cI$-space monoids and we define 
$EG \to BG$
to be the positive fibration of commutative $\cI$-space monoids 
resulting from a (functorial) factorization 
$
\xymatrix@-1pc{B(U^{\cI},G,G)
  \ar@{ >->}[r]^-{\sim} & EG \ar@{->>}[r] & BG
}
$
in the positive flat model structure on $\cC\Top^{\cI}$. 
\end{definition}

The inclusion of the last copy of $G$ in $B(U^{\cI},G,G)$ and the acyclic cofibration $B(U^{\cI},G,G) \to EG$ from the above definition make $EG$ a commutative $G$-algebra. 
Hence we can view $EG \to BG$ as a map of commutative $G$-algebras, where the structure map 
$G \to BG$ factors through the terminal map $G \to U^{\cI}$. It follows from general properties of the two-sided bar construction that $B(U^{\cI},G,G)$ contains $U^{\cI}$ as a deformation retract 
(see \cite[Section~9]{May-geometry}) which implies that the unit $U^{\cI}\to EG$ is an $\cI$-equivalence.

\begin{remark}
If $G$ and $G'$ are cofibrant in the positive flat model structure on $\cC\Top^{\cI}$, then an $\cI$-equivalence $G\to G'$ induces an $\cI$-equivalence 
$BG \to BG'$.  Therefore $BG$ represents a well-defined homotopy type if $G$ is cofibrant in $\cC\Top^{\cI}$.
  The fact that $EG \to BG$ is a positive fibration by construction will free us from making additional fibrancy assumptions later on and it is for this reason we prefer to work with $EG$ instead of $B(U^{\cI},G,G)$. Such additional fibrancy conditions were needed in Lewis' work on the ``classical'' Thom spectrum functor \cite[Section~IX]{LMS}  since the model categorical techniques for making multiplicative fibrant replacements were not in place at the time when that paper was written. In our setting, additional fibrancy conditions will only be needed to make the passage from space level data to $\cI$-space data in Section~\ref{sec:Thom-spectra-from-space-level-data} homotopy invariant. 
\end{remark}

Let again $G$ be a commutative $\cI$-space monoid, and let $\Top^{\cI}_G$ denote the category of (right) $G$-modules with respect to the $\boxtimes$-product. The category $\Top^{\cI}_G$ inherits a
symmetric monoidal product $X\boxtimes_GY$ defined by the usual coequalizer diagram
and with $G$ as the monoidal unit. Identifying the category of $U^{\cI}$-modules with $\Top^{\cI}$, the map 
$G \to U^{\cI}$ induces a restriction of scalars functor $\mathrm{triv}_G \colon \Top^{\cI} \to \Top^{\cI}_G$ whose left adjoint
is the strong symmetric monoidal extension of scalars functor
$-\boxtimes_GU^{\cI}\colon \Top^{\cI}_G\to \Top^{\cI}$. In the following we shall use the notation $EG \to \mathrm{triv}_G BG$ when we think of the map of $G$-algebras in Definition~\ref{def:EG-BG} as a map of 
$G$-modules. 

\begin{definition}\label{def:VU-adjunction} Let $G$ be a commutative $\cI$-space monoid. We define a pair of adjoint functors $(V,U)$ by the composition 
\begin{equation}\label{eq:VU-adjunction}
V\colon \Top^{\cI}_G/EG \rightleftarrows \Top^{\cI}_G/\mathrm{triv}_G BG  \rightleftarrows \Top^{\cI}/BG : U
\end{equation}
where the first adjunction is given by composition with and base change along $EG \to \mathrm{triv}_G BG$ and the second adjunction is induced by the adjoint functors $-\boxtimes_GU^{\cI}$ and $\mathrm{triv}_G$. 
\end{definition}
More explicitly, the left adjoint $V$ sends $X \to EG$ to the composition
\[
X\boxtimes_GU^{\cI}\to  EG\boxtimes_GU^{\cI} \to (\mathrm{triv}_G BG)\boxtimes_G U^{\cI} \cong BG
\] 
and the right adjoint $U$ sends $Y \to BG$ to the pullback of the diagram 
\[
EG \to \mathrm{triv}_G BG \ot \mathrm{triv}_GY.
\]
Since $BG$ is a commutative $\cI$-space monoid, the category $\Top^{\cI}/BG$ inherits a symmetric monidal structure from $\Top^{\cI}$: Given maps of $\cI$-spaces $\alpha\colon X\to BG$ and $\beta\colon Y\to BG$, the
monoidal product $\alpha\boxtimes \beta$ is defined by the composition
\[
\alpha\boxtimes \beta\colon X\boxtimes Y\to BG\boxtimes BG\to BG.
\]
This product has the unit $\iota\colon U^{\cI}\to BG$ as its monoidal unit. Notice also that a commutative monoid in the symmetric monoidal category $(\Top_G^{\cI},\boxtimes_G,G)$ is the same thing as a commutative $G$-algebra. Hence we may view $EG$ as a commutative monoid in $\Top_G^{\cI}$ so that $\Top_G^{\cI}/EG$ inherits a symmetric monoidal structure in the same way. Here the monoidal unit is the $G$-algebra unit $\iota_G\colon G\to EG$.

\begin{lemma}\label{lem:ISG-ISBG-adj-monoidal}
The left adjoint $V$ in~\eqref{eq:VU-adjunction} is a strong symmetric monoidal functor, and its right adjoint $U$ is a lax symmetric monoidal functor.  
\end{lemma}
\begin{proof}
This is standard, but we later need to know the details. The fact that the extension of scalars functor $-\boxtimes_GU^{\cI}$ is strong symmetric monoidal implies that the same holds for $V$. Using this, the lax symmetric monoidal structure maps for $U$ are given by the compositions
\[
U(\alpha)\boxtimes_G U(\beta)\to UV(U(\alpha)\boxtimes_G U(\beta))\xleftarrow{\iso} U(VU(\alpha)\boxtimes VU(\beta))
\to U(\alpha\boxtimes \beta)
\]
and 
$
\iota_G \to UV(\iota_G)\stackrel{\cong}{\ot}U(\iota),
$
induced by the adjunction unit and counit and the monoidal structure map for~$V$. (These structure maps can also be described explicitly by the universal property of the pullback.)
\end{proof}

Having arranged for $EG \to BG$ to be a fibration in the positive flat model structure, right properness of the latter model structure \cite[Proposition~11.3]{Sagave-S_diagram} has the following consequence. 
\begin{lemma}\label{lem:U-preserves-I-equiv}
The functor $U \colon \Top^{\cI}/BG \to \Top^{\cI}_G/EG$ preserves $\cI$-equivalences. \qed
\end{lemma}

It follows from \cite[Proposition 3.10]{Sagave-S_diagram} and~\cite[Theorem~4.1(2)]{Schwede_S-algebras} that the absolute and positive flat model structures on 
$\Top^{\cI}$ lift to corresponding absolute and positive flat model structures on $\Top_G^{\cI}$ in the sense that a map of 
$G$-modules is a weak equivalence or fibration if and only if the underlying map of $\cI$-spaces is so. The next theorem shows  that under suitable assumptions on $G$ we may view $EG\to BG$ as a universal fibration and $BG$ as a classifying space for 
$G$-modules.  

\begin{theorem}\label{thm:Top^I_G-Quillen-equivalence}
Let $G$ be a grouplike and cofibrant commutative $\cI$-space monoid. Then the $(V,U)$-adjunction \eqref{eq:VU-adjunction} and the adjunction induced by the canonical map $EG\to U^{\cI}$ define a chain of Quillen equivalences
\[
\Top^{\cI}_G\leftrightarrows \Top^{\cI}_G/EG\rightleftarrows \Top^{\cI}/BG
\]  
with respect to the absolute and positive flat model structures on these categories.
\end{theorem}
\begin{proof}
The first Quillen adjunction is given by composition with and pullback along the map $EG\to U^{\cI}$ and is a Quillen equivalence because the latter is an \mbox{$\cI$-equivalence}. We postpone the argument why the $(V,U)$-adjunction defines a Quillen equivalence until Section~\ref{sec:modules-classifying-spaces} where the statement appears as part of Proposition~\ref{prop:ISG-ISBG-adj}. 
\end{proof} 



Let $\alpha\colon X\to BG$ and $\beta\colon Y\to BG$ be maps of $\cI$-spaces with $X$ or $Y$ flat, and let $U(\alpha)^{\mathrm{cof}}\to U(\alpha)$  and 
$U(\beta)^{\mathrm{cof}}\to U(\beta)$ be cofibrant replacements in the absolute flat model structure on $\Top_G^{\cI}/EG$. We define the \emph{derived monoidal multiplication for $U$} to be the composition
\begin{equation}\label{eq:U-derived-monoidal}
U(\alpha)^{\mathrm{cof}}\boxtimes_G U(\beta)^{\mathrm{cof}}\to U(\alpha)\boxtimes_G U(\beta)\to U(\alpha\boxtimes \beta).
\end{equation}
Such a map can be defined without conditions on $X$ and $Y$, but the term \emph{derived monoidal multiplication} is only justified if either $X$ or $Y$ is flat.
\begin{lemma}\label{lem:U-derived-monoidal}
Let $G$ be a grouplike and cofibrant commutative $\cI$-space monoid. Then the monoidal unit $\iota_G \to U(\iota)$ and the derived monoidal multiplication \eqref{eq:U-derived-monoidal} are $\cI$-equivalences.
\end{lemma} 
\begin{proof}
The fact that the $(V,U)$-adjunction defines a Quillen equivalence implies that the derived unit and counit of the adjunction are $\cI$-equivalences (see \cite[Proposition~1.3.13]{Hovey_model}). Using the description of the monoidal structure maps given in the proof of Lemma~\ref{lem:ISG-ISBG-adj-monoidal}, the result then follows from the homotopy invariance of $U$ stated in Lemma~\ref{lem:U-preserves-I-equiv} and the assumption that $X$ or $Y$ is flat.
\end{proof} 
 
\section{Thom spectra from \texorpdfstring{$\cI$}{I}-space data}\label{sec:gen-Thom-spectra}
Our definition of the Thom spectrum functor is based in part on the two-sided bar construction, and we begin by reviewing the basic properties of the latter.

\subsection{The two-sided bar construction for symmetric spectra}
We shall work with the \emph{positive} and \emph{absolute flat stable model structures} on $\Sp^{\Sigma}$ introduced in \cite{Shipley_convenient} and \cite{Schwede_SymSp}. These model structures have the same stable equivalences as the standard stable model structure (see \cite{HSS} and \cite{MMSS}), but the flat versions have more cofibrations which makes them better suited for certain applications. Most important for our purposes is the fact that the positive flat stable model structure lifts to a (positive flat) stable model structure on the category of commutative symmetric ring spectra $\cC\Sp^{\Sigma}$, and that if a commutative symmetric ring spectrum is cofibrant in this model structure, then 
its underlying symmetric spectrum is cofibrant in the absolute flat stable model structure on $\Sp^{\Sigma}$. We shall use the term \emph{flat symmetric spectrum} to mean a cofibrant object in the absolute flat stable model structure on $\Sp^{\Sigma}$.

Given a commutative symmetric ring spectrum $R$, we write $\Sp^{\Sigma}_R$ for the symmetric monoidal category of (right) $R$-modules under the 
usual $R$-balanced smash product $\wedge_R$. The absolute flat stable model structure on $\Sp^{\Sigma}$ lifts to a flat stable model structure on $\Sp^{\Sigma}_R$ in which a map of $R$-modules is a stable equivalence or fibration if and only if the underlying map of symmetric spectra is so with respect to the absolute flat stable model structure. We shall use the term \emph{flat $R$-module} for a cofibrant object in this model structure. It is a useful fact that if $M$ is a flat $R$-module, then the smash product $M\wedge_R-$ preserves stable equivalences between general $R$-modules without further cofibrancy conditions (see e.g.\ the proof of \cite[Lemma~4.8]{Rognes-S-S_log-THH} and the analogous argument for $\cI$-spaces in Lemma~\ref{lem:boxtimes_M-inv} below).

Now let $P\to R$ be a map of commutative symmetric ring spectra and recall that for a $P$-module $M$, the two-sided bar construction $B(M,P,R)$ is the geometric realization of the simplicial symmetric spectrum $[n]\mapsto M\wedge P^{\wedge n}\wedge R$ with simplicial structure maps defined as in 
\cite[Section 10]{May-geometry}. This construction gives a lax symmetric monoidal functor
\[
B(-,P,R)\colon \Spsym{P} \to \Spsym{R}
\]
with monoidal structure maps 
\begin{equation}\label{eq:bar-monoidal}
B(M,P,R)\wedge_RB(N,P,R)\to B(M\wedge_PN,P,R)\quad \text{and} \quad R\to B(P,P,R)
\end{equation}
induced by the multiplicative structures of $P$ and $R$ and the unit of $P$, cf.\ \cite[Lemma 10.1]{May-geometry}. The two-sided bar construction is related to the actual smash product by a natural symmetric monoidal map $B(M,P,R)\to M\wedge_PR$ induced by the canonical map $M\wedge R\to M \wedge_PR$ in simplicial degree zero. Applying the argument from
\cite[Lemma~4.1.9]{Shipley_THH} to the case of a flat $P$-module, we get the following result.

\begin{lemma}\label{lem:bar-smash-equivalence}
If $M$ is a flat $P$-module, then $B(M,P,R)\to M\wedge_PR$ is a stable equivalence.\qed
\end{lemma}
The advantage of the two-sided bar construction compared to the extension of scalars functor $-\wedge_PR$ is that the former is homotopically well-behaved under less restrictive cofibrancy conditions as we shall see below. Let us say that a commutative symmetric ring spectrum $P$ has a \emph{flat unit} if the unit $\bS\to P$ is a cofibration in the absolute flat stable model structure on $\Sp^{\Sigma}$.

\begin{lemma}\label{lem:invariant-bar}
Let $P\to R$ be a map of commutative symmetric ring spectra and suppose that $P$ has a flat unit. 
\begin{enumerate}
\item[(i)]
If $M$ is a $P$-module such that the underlying symmetric spectrum of $M$ is flat, then $B(M,P,R)$ is a flat $R$-module.
\item[(ii)]
If the underlying symmetric spectrum of $R$ is flat, then the functor defined by $B(-,P,R)$ preserves stable equivalences. 
\end{enumerate}
\end{lemma}

\begin{proof}
For the proof of (i) we consider the skeletal filtration of $B(M,P,R)$ inherited from the underlying simplicial object. (We refer 
to \cite[Corollary~2.4]{Lewis_when-cofibration} and its proof for a discussion of skeletal filtrations).
The assumption on $M$ implies that the \mbox{$0$-skeleton} $M\wedge R$ is a flat $R$-module. Furthermore, using the pushout-product axiom, the assumption that $P$ has a flat unit implies that the inclusion of the $(n-1)$-skeleton in the $n$-skeleton is a cofibration in the flat stable module structure on $\Spsym{R}$ for all $n$. This implies that $B(M,P,R)$ is a flat $R$-module.

As for the claim in (ii), let $M\to M'$ be a stable equivalence of $P$-modules, and notice that the induced map of $R$-modules admits a factorisation
\[
B(M,P,R)\cong M\wedge_PB(P,P,R)\to M'\wedge_PB(P,P,R)\cong B(M',P,R).
\]
Using an argument similar to that used in (i) (or that $B(P,P,R)$ is isomorphic to $B(R,P,P)$), the assumption on $R$ implies that $B(P,P,R)$ is a flat $P$-module. Hence the functor $-\wedge_PB(P,P,R)$ preserves stable equivalences and the result follows.
\end{proof}

Whereas the extension of scalars functor $-\wedge_PR$ is strong symmetric monoidal, the two-sidet bar construction only gives a lax symmetric monoidal functor. 

\begin{lemma}\label{lem:bar-monoidal-equivalence}
Let $P\to R$ be a map of commutative symmetric ring spectra and suppose that $P$ is cofibrant in the positive flat stable model structure on $\cC\Sp^{\Sigma}$. Then the monoidal structure maps in \eqref{eq:bar-monoidal} are stable equivalences provided that $M$ and $N$ are flat $P$-modules.
\end{lemma}
\begin{proof}
The assumption on $P$ implies that its underlying symmetric spectrum is flat which in turn implies that every flat $P$-module has underlying flat symmetric spectrum (see \cite{Shipley_convenient}*{Section~4}). It also follows that $P$ has a flat unit so that $B(M,P,R)$ and $B(N,P,R)$ are flat $R$-modules by Lemma~\ref{lem:invariant-bar}. Hence the monoidal structure maps in \eqref{eq:bar-monoidal} are stably equivalent to the monoidal structure maps for the extension of scalars functor $-\wedge_PR$ that we know to be isomorphisms.
\end{proof}

\subsection{The Thom spectrum functor on \texorpdfstring{$\Top^{\cI}/BG$}{TopI/BG}}\label{subsec:Thom-on-Top^I/BG}
Let $R$ be a semistable commutative symmetric ring spectrum and let $G\to\GL_1^{\cI}(R)$
 be a cofibrant replacement of its units in the
  positive flat model structure on $\cC\Top^{\cI}$. We consider
  the two-sided bar construction $B(-, \bS^{\cI}[G],R)$ associated
  with the canonical map $\bS^{\cI}[G]\to R$ in
  \eqref{eq:SI-GLoneIR-to-R}, and write $M\GL_1^{\cI}(R)$ for the commutative $R$-algebra spectrum $B(\bS^{\cI}[EG],
  \bS^{\cI}[G],R)$ where $EG$ is the commutative $G$-algebra introduced in Definition~\ref{def:EG-BG}.
\begin{definition}\label{def:TI}
  The $R$-module Thom spectrum functor $T^{\cI}$ is the composition
  \[
  T^{\cI}\colon \Top^{\cI}/BG \xrightarrow{U} \Top^{\cI}_G/EG \xrightarrow{\bS^{\cI}}
  \Spsym{\bS^{\cI}[G]}/\bS^{\cI}[EG]\xrightarrow{B(-, \bS^{\cI}[G],R)}
  \Spsym{R}/M\GL_1^{\cI}(R)
  \] 
  of the right adjoint $U$ from \eqref{eq:VU-adjunction} and the functors of
  over-categories induced by $\bS^{\cI}$ and $B(-,\bS^{\cI}[G],R)$.
\end{definition}

For $R = \bS$, one can show by a direct comparison that the resulting Thom spectrum functor is equivalent to that considered by Lewis and Mahowald. 
This implies in particular that $M\GL_1^{\cI}(\bS)$ is stably equivalent to the Thom spectrum usually denoted~$MF$.

\begin{remark}\label{rem:alternative-setting}
For definiteness, we have chosen the setting of symmetric spectra of topological spaces for our generalized Thom spectra, but it is also possible to translate the constructions and results in the paper to the setting of orthogonal spectra \cite{MMSS}. The main point in this translation is to replace the category of $\cI$-spaces with the corresponding diagram category of $\mathcal V$-spaces, where $\mathcal V$ denotes the topological category with objects the standard real inner product spaces $\mathbb R^n$ and morphisms the linear isometries. Lind's work on $\mathcal V$-spaces~\cite{Lind_diagram} and Stolz' flat model structure for orthogonal 
spectra~\cite{Brun-D-S_equivariant} supply many of the necessary technical foundations.  Working with orthogonal spectra has the technical advantage that stable equivalences induce isomorphisms on spectrum homotopy groups so that the semistability condition becomes superfluous.

One can also modify the constructions to obtain a generalized Thom
spectrum functor for symmetric spectra in simplicial sets. In this
case, we have to assume that $R$ is both level fibrant and semistable
for $\GL^{\cI}_1(R)$ to capture the desired homotopy type. This is
useful even in the classical case because the approach to Thom spectra
for maps to $BF$ in~\cite{Schlichtkrull_Thom-symmetric} does not seem
to have a simplicial counterpart with good monoidal properties.
\end{remark}

The homotopy invariance statement in the next proposition is the main reason why we prefer to work with the two-sided bar construction instead of the actual smash product $-\wedge_{\bS^{\cI}[G]}R$. 
We recall that if $R$ is cofibrant in the positive flat stable model structure on $\cC\Sp^{\Sigma}$, then its underlying symmetric spectrum is automatically flat.

\begin{proposition}\label{prop:hty-inv-of-TI}
Suppose that the underlying symmetric spectrum of $R$ is flat. Then $T^{\cI}$ takes $\cI$-equivalences over $BG$ to stable equivalences over $M\GL_1^{\cI}(R)$. 
\end{proposition}
\begin{proof}
The functor $U$ preserves $\cI$-equivalences by Lemma~\ref{lem:U-preserves-I-equiv}, the functor $\bS^{\cI}$ takes $\cI$-equivalences to stable equivalences by  Lemma~\ref{lem:hty-inv-of-SI-OmegaI}, and $B(-,\bS^{\cI}[G],R)$ preserves stable equivalences by Lemma~\ref{lem:invariant-bar} and the assumption on $R$. 
\end{proof}

\begin{remark}
  The above construction also leads to an $R$-module Thom spectrum functor
  for associative (not necessarily commutative) symmetric ring spectra $R$. The homotopy invariance property continues to hold, but there are no monoidal structures on $\Top^{\cI}/BG$ and $\Spsym{R}$ that can be
  preserved by such a Thom spectrum functor.
\end{remark}

Next we turn to the monoidal properties of $T^{\cI}$. Let 
$\iota \colon U^{\cI} \to BG$ denote the unit of the commutative $\cI$-space monoid $BG$.

\begin{proposition}\label{prop:TI-lax-monoidal}
The functor $T^{\cI}$ is lax symmetric monoidal with monoidal structure maps given by maps of $R$-modules
\[
  T^{\cI}(\alpha)\sm_R T^{\cI}(\beta) \to T^{\cI}(\alpha\boxtimes\beta) \quad \text{ and }\quad R \to
  T^{\cI}(\iota)
\]
satisfying the usual associativity, commutativity, and unitality conditions.
\end{proposition}
\begin{proof}
The functor $U$ is lax symmetric monoidal by Lemma~\ref{lem:ISG-ISBG-adj-monoidal}, the functor  
$\bS^{\cI}$ is even strong symmetric monoidal, and the two-sided bar construction $B(-,\bS^{\cI}[G],R)$ is lax symmetric monoidal as we noted above. Hence the composition is also lax symmetric monoidal.
\end{proof}

Let $\cD$ be an operad in spaces. Since $BG$ and $M\GL_1^{\cI}(R)$ are commutative monoids in the symmetric monoidal categories $\Top^{\cI}$ and $\Spsym{R}$, restricting along the canonical operad morphism from $\cD$ to the (terminal) commutativity operad allows us to view $BG$ and $M\GL_1^{\cI}(R)$ as $\cD$-algebras. Writing $\Top^{\cI}[\cD]$ and 
$\Spsym{R}[\cD]$ for the categories of $\cD$-algebras in $\Top^{\cI}$ and $\Spsym{R}$, 
Proposition~\ref{prop:TI-lax-monoidal} then has the following corollary.
\begin{corollary}\label{cor:TI-preserves-operad-actions}
Let $\cD$ be an operad in spaces. Then $T^{\cI}$ induces a functor 
\[
T^{\cI}\colon \Top^{\cI}[\cD]/BG \to \Spsym{R}[\cD]/M\GL_1^{\cI}(R)  
\]
on the categories of $\cD$-algebras over $BG$ and $M\GL_1^{\cI}(R)$.\qed
\end{corollary}

We proceed to analyze the homotopical properties of the monoidal structure maps for $T^{\cI}$. Let $\alpha\colon
X \to BG$ and $\beta\colon Y\to BG$ be maps of $\cI$-spaces with
$X$ or $Y$ flat, and let $T^{\cI}(\alpha)^{\mathrm{cof}}\to T^{\cI}(\alpha)$ and
$T^{\cI}(\beta)^{\mathrm{cof}} \to T^{\cI}(\beta)$ be cofibrant
replacements in the flat stable model structure on $\Sp^{\Sigma}_R$. Composing the monoidal multiplication from 
Proposition~\ref{prop:TI-lax-monoidal} with these cofibrant replacements, we get a map of
$R$-modules 
\begin{equation}\label{eq:derived-monoidal-str-maps}
  T^{\cI}(\alpha)^{\mathrm{cof}}\sm_R T^{\cI}(\beta)^{\mathrm{cof}}
\to T^{\cI}(\alpha)\sm_R T^{\cI}(\beta)\to T^{\cI}(\alpha\boxtimes\beta)
  \end{equation}
  that we refer to as the \emph{derived
  monoidal multiplication of $T^{\cI}$}.

\begin{proposition}\label{prop:TI-derived-monoidal}
 Let $R$ be a commutative symmetric ring spectrum with underlying flat symmetric spectrum. Then the monoidal unit 
 $R\to T^{\cI}(\iota)$ and the derived monoidal multiplication \eqref{eq:derived-monoidal-str-maps} are stable equivalences.
   
\end{proposition}
\begin{proof}
We can write the monoidal unit as the composition
\[
R\to B(\bS^{\cI}[G],\bS^{\cI}[G], R)\to B(\bS^{\cI}[U(\iota)],\bS^{\cI}[G], R)
\]
where the first map is a stable equivalence by general properties of the two-sided bar construction, and the second map is a stable equivalence since $\bS^{\cI}[G]\to \bS^{\cI}[U(\iota)]$ is a stable equivalence by Lemmas~\ref{lem:hty-inv-of-SI-OmegaI} and \ref{lem:U-derived-monoidal}. 

In order to analyze the derived monoidal multiplication, we first choose cofibrant replacements $U(\alpha)^{\mathrm{cof}}\to U(\alpha)$ and $U(\beta)^{\mathrm{cof}}\to U(\beta)$ as in Lemma~\ref{lem:U-derived-monoidal}. Applying the functor
$B(\bS^{\cI}[-], \bS^{\cI}[G],R)$ to these cofibrant replacements gives us cofibrant replacements of $T^{\cI}(\alpha)$ and $T^{\cI}(\beta)$ by flat $R$-modules as follows from Lemma~\ref{lem:invariant-bar}. Furthermore, combining Lemmas~\ref{lem:U-derived-monoidal} and \ref{lem:bar-monoidal-equivalence}, we see that the derived monoidal multiplication is a stable equivalence for these particular choices of cofibrant replacements. That the same holds for all choices of cofibrant replacements now follows from the fact that the smash product $\wedge_R$ preserves stable equivalences between flat $R$-modules.
\end{proof}

Recall that the topological categories $\Top^{\cI}$ and $\Sp^{\Sigma}_R$ are tensored over $\Top$ with tensors defined by the levelwise cartesian products $X\times Q$  and the levelwise smash product $E\wedge Q_+$ for a space $Q$, an $\cI$-space $X$, and a symmetric spectrum $E$. Precomposing with the appropriate projections, we get induced tensor structures on the over-categories $\cS^{\cI}/BG$ and $\Spsym{R}/M\GL_1^{\cI}(R)$. The next proposition states that the Thom spectrum functor preserves these tensors.
\begin{proposition}\label{prop:TI-preserves-tensors}
For a map of $\cI$-spaces $\alpha \colon X \to BG$ and a space $Q$, there is a natural isomorphism 
$ T^{\cI}(\alpha\times Q)\iso T^{\cI}(\alpha)\wedge Q_+$. \qed
\end{proposition}

\subsection{Thom spectra from \texorpdfstring{$\cI$}{I}-space monoids over \texorpdfstring{$G$}{G}}
Let $R$ be a (semistable as always) commutative symmetric ring spectrum with flat unit and underlying flat symmetric spectrum, and let again $G\to \GL_1^{\cI}(R)$ be a cofibrant replacement.  Recall from \cite[Proposition~9.3]{Sagave-S_diagram} that the absolute flat model structure on $\Top^{\cI}$ lifts to an (absolute flat) model structure on the category $\cA\Top^{\cI}$ of (not necessarily commutative) $\cI$-space monoids. We shall use the term \emph{cofibrant $\cI$-space monoid} for a cofibrant object in this model structure on $\cA\Top^{\cI}$. The induced over-category model structures on $\cA\Top^{\cI}/G$ and $\cA\Top^{\cI}/\GL_1^{\cI}(R)$ are Quillen equivalent, and it will be most convenient to work directly with $\cI$-space monoids over $G$. Given a map of $\cI$-space monoids $\alpha\colon H\to G$, we write $B\alpha\colon BH\to BG$ for the induced map of bar constructions.

\begin{proposition}\label{prop:B-alpha-Thom}
Let $H$ be a grouplike and cofibrant $\cI$-space monoid, and let $\alpha\colon H\to G$ be a map of $\cI$-space monoids. Then there is a chain of natural stable equivalences 
\[
T^{\cI}(B\alpha)\simeq B(R,R\wedge \bS^{\cI}[H],R^\alpha)
\] 
where the copy of $R$ on the left in the two-sided bar construction has module structure induced by the projection $R\wedge \bS^{\cI}[H]\to R$, and $R^\alpha$ has $R$ as its underlying symmetric spectrum and module structure induced by the map  $H\stackrel{\alpha}{\to} G\to \GL_1^{\cI}(R)$ and the multiplication in $R$.
\end{proposition}

On the level of homotopy groups we have $\pi_*(R\wedge \bS^{\cI}[H])\cong R_*(H_{h\cI})$ with Pontryagin ring structure induced by the monoid structure of $H_{h\cI}$. Since the two-sided bar construction $B(R,R\wedge \bS^{\cI}[H],R^\alpha)$ represents the derived smash product of $R$ and $R^\alpha$ over $R\wedge \bS^{\cI}[H]$, we thus get the following $\Tor$ spectral sequence (compare e.g.\ \cite[Theorem IV.4.1]{EKMM}): 
\[
E^2_{*,*}=\Tor_*^{R_*(H_{h\cI})}(R_*,R_*^{\alpha})\Longrightarrow \pi_*(T^{\cI}(B\alpha))
\]

\begin{proof}[Proof of Proposition~\ref{prop:B-alpha-Thom}] The map $\alpha$ gives rise to a commutative diagram of $\cI$-spaces
\[
\xymatrix@-1pc{ B(U^{\cI},H,G) \ar[r] \ar[d]& B(U^{\cI},G,G)\ar[d]\\
BH \ar[r]^{B\alpha} & BG 
}
\]
which we claim to be homotopy cartesian. Thus, we must show that applying the homotopy colimit functor $(-)_{h\cI}$ we get a homotopy cartesian diagram of spaces (see \cite[Corollary~11.4]{Sagave-S_diagram}). For this it suffices to show that replacing $G$ by $G_{h\cI}$, $H$ by $H_{h\cI}$, and the $\boxtimes$-product by the cartesian product of spaces, the diagram becomes homotopy cartesian. In this situation the diagram is actually a pullback diagram and the vertical maps are quasifibrations by \cite[Theorem~7.6]{May-classifying} since $H_{h\cI}$ and $G_{h\cI}$ are grouplike. This shows that the diagram is homotopy cartesian. By definition of the universal fibration $EG\to BG$ this in turn implies that the induced map 
$
B(U^{\cI},H,G)\to U(B\alpha)
$
is an $\cI$-equivalence, and hence that applying $\bS^{\cI}$ gives a stable equivalence of $\bS[G]$-modules
\[
B(\bS,\bS^{\cI}[H],\bS^{\cI}[G])\cong \bS^{\cI}[B(U^{\cI},H,G)]\stackrel{\simeq}{\to} \bS^{\cI}[U(B\alpha)].
\]
The chain of stable equivalences
\[
\begin{gathered}
B(\bS^{\cI}[U(B\alpha)],\bS^{\cI}[G],R)\xl{\simeq} B(B(\bS,\bS^{\cI}[H],\bS^{\cI}[G]),\bS^{\cI}[G],R)\\
\xr{\simeq} B(\bS, \bS^{\cI}[H],\bS^{\cI}[G])\wedge_{\bS^{\cI}[G]}R  \xr{\simeq}B(\bS,\bS^{\cI}[H],R^{\alpha})\xr{\simeq} B(R,R\wedge\bS^{\cI}[H],R^{\alpha})
\end{gathered}
\]
then gives the statement in the proposition. For the second equivalence we argue as in the proof of Lemma~\ref{lem:invariant-bar}(i) to show that the $\bS^{\cI}[G]$-module $B(\bS, \bS^{\cI}[H],\bS^{\cI}[G])$ is flat such that Lemma~\ref{lem:bar-smash-equivalence} applies. The flatness assumption on the unit of $R$ ensures that the degeneracy maps in the simplicial spectrum underlying $B(R,R\wedge\bS^{\cI}[H],R^{\alpha})$ are cofibrations so the geometric realization is homotopically well-behaved.  
\end{proof}

\begin{remark}
In general, given a map of based $\cI$-spaces $\beta\colon X\to BG$ with $X_{h\cI}$ path connected, one can show that there exists a grouplike and cofibrant $\cI$-space monoid $M$ and a map of $\cI$-space monoids $\alpha\colon M\to G$ such that $B\alpha$ and $\beta$ are weakly equivalent as objects of $\Top^{\cI}/BG$.  It follows that the description of the Thom spectrum functor in Proposition~\ref{prop:B-alpha-Thom} can be extended to all such maps of based 
$\cI$-spaces $X\to BG$ with $X_{h\cI}$ path connected.   
\end{remark}

\subsection{Thom spectra over suspensions}\label{subsec:Thom-I-suspension}
Let $R$ and $G$ be as above and suppose that the underlying symmetric spectrum of $R$ is flat. (We don't need to assume that $R$ has a flat unit for the results in this section). 
Let $X$ be a based $\cI$-space (that is, an $\cI$-space equipped with a map $U^{\cI}\to X$), which we assume to be levelwise well-based. We write $CX$ for the reduced cone (the levelwise smash product with the unit interval $I$ based at $0$) and $\Sigma X$ for the reduced suspension (the levelwise smash product with $S^1=I/\partial I$). Given a map of based $\cI$-spaces $\alpha\colon X\to G$, we let $\Sigma\alpha$ be the composition 
\[
\Sigma\alpha\colon \Sigma X\to \Sigma G\to BG,
\] 
where the second map is the inclusion of $\Sigma G$ as the 1-skeleton of $BG$.

\begin{proposition}\label{prop:Thom-suspension}
The Thom spectrum $T^{\cI}(\Sigma\alpha)$ fits in a functorial homotopy cocartesian diagram of $R$-modules
\[
\xymatrix@-1pc{
\bS^{\cI}[X]\wedge R \ar[r]\ar[d] & \bS^{\cI}[CX]\wedge R\ar[d]\\
B(\bS^{\cI}[G],\bS^{\cI}[G],R)\ar[r] & T^{\cI}(\Sigma \alpha)
}
\]
where the upper horizontal map is induced by the inclusion of $X$ in $CX$, and the vertical map on the left is induced by $\alpha$ and the inclusion of $\bS^{\cI}[G]\wedge R$ as the $0$-skeleton in $B(\bS^{\cI}[G],\bS^{\cI}[G],R)$
\end{proposition}

Notice that the composition of the vertical map on the left with the stable equivalence $B(\bS^{\cI}[G],\bS^{\cI}[G],R)\xr{\simeq}R$ can be described as multiplication by $\bS^{\cI}[\alpha]$ via the map $G\to \GL_1^{\cI}(R)$. Hence the diagram in the proposition is stably equivalent to the following diagram in the stable homotopy category
\[
\xymatrix@-1pc{ R\wedge (X_{h\cI})_+ \ar[r]\ar[d]  & R\ar[d]\\
R \ar[r] & T^{\cI}(\Sigma \alpha)
}
\]
where on the left we compose $\alpha_{h\cI}$ with the map $G_{h\cI}\to \GL_1^{\cI}(R)_{h\cI}$ and use the multiplication in $R$. Consequently we get a homotopy cofiber sequence of $R$-modules
\[
R\wedge X_{h\cI}\to R\to T^{\cI}(\Sigma\alpha).
\]
We prefer the description of $T^{\cI}(\Sigma\alpha)$ in the proposition since it has the advantage of being strictly functorial.
The proof of the proposition is based on the next lemma.

\begin{lemma}\label{lem:suspension-diagram}
There is a functorial homotopy cocartesian diagram of $G$-modules
\[
\xymatrix@-1pc{
X\boxtimes G \ar[r]\ar[d] & CX\boxtimes G\ar[d]\\
G \ar[r] & U(\Sigma\alpha)
}
\]
where the upper horizontal map is induced by the inclusion of $X$ in $CX$, and the vertical map on the left is induced by $\alpha$ and the multiplication in $G$.
\end{lemma}
\begin{proof}
We first observe that the 1-skeleton of $B(U^{\cI},G,G)$ can be identified with the pushout of the diagram
$
G\ot G\boxtimes G\to CG\boxtimes G
$,
where the maps are given by the multiplication in $G$ and the inclusion of $G$ in $CG$. Hence $\Sigma \alpha$ fits as the composition in the bottom line of the commutative diagram
\[
\xymatrix@-.5pc{
G\cup_{X\boxtimes G}CX\boxtimes G \ar[r]\ar[d] & G\cup_{G\boxtimes G}CG\boxtimes G \ar[r]\ar[d] & B(U^{\cI},G,G)\ar[d]\\
\Sigma X \ar[r] & \Sigma G \ar[r] & BG 
}
\]
which in turn gives rise to the commutative diagram in the lemma. Furthermore, by the definition of the universal fibration $EG\to BG$, the statement in the lemma is equivalent to the outer diagram being homotopy cartesian. Now recall from \cite[Corollary~11.4]{Sagave-S_diagram}) that a commutative diagram of $\cI$-spaces is homotopy cartesian if and only if it becomes a homotopy cartesian diagram of spaces when passing to homotopy colimits. Thus, it suffices to show that replacing $X$ by $X_{h\cI}$, $G$ by $G_{h\cI}$, and the $\boxtimes$-product by the cartesian diagram of spaces, the outer diagram becomes homotopy cartesian. In this situation one can check that the outer diagram is in fact a pullback diagram. Indeed, it is the geometric realization of a pullback diagram of simplicial spaces, where the simplicial cone is the smash product with the standard simplicial one-simplex $\Delta[1]$, and the simplicial suspension is the smash product with the simplicial circle $\Delta[1]/\partial\Delta[1]$.   
The standard arguments which show that $B(*,G_{h\cI},G_{h\cI})\to BG_{h\cI}$ is a quasifibration also show that the projection 
\[
G_{h\cI}\cup_{X_{h\cI}\times G_{h\cI}}(C(X_{h\cI})\times G_{h\cI})\to \Sigma(X_{h\cI})
\] 
is a quasifibration since $G_{h\cI}$ is grouplike. This gives the result. 
\end{proof}

\begin{proof}[Proof of Proposition~\ref{prop:Thom-suspension}]
Applying the functor $B(\bS^{\cI}[-],\bS^{\cI}[G],R)$ to the homotopy cocartesian diagram in Lemma~\ref{lem:suspension-diagram}, we get a commutative diagram of $R$-modules
\[
\xymatrix@-0.5pc{
B(\bS^{\cI}[X]\wedge\bS^{\cI}[G],\bS^{\cI}[G],R) \ar[r] \ar[d]& B(\bS^{\cI}[CX]\wedge\bS^{\cI}[G],\bS^{\cI}[G],R)\ar[d]\\
B(\bS^{\cI}[G],\bS^{\cI}[G],R)\ar[r] & B(\bS^{\cI}[U(\Sigma\alpha)],\bS^{\cI}[G],R)
}
\]
which is homotopy cocartesian by Lemmas~\ref{lem:hty-inv-of-SI-OmegaI} and \ref{lem:invariant-bar}. Now the assumption that the underlying symmetric spectrum of $R$ be flat implies that there is natural stable equivalence
\[
-\wedge R\to B(-\wedge \bS^{\cI}[G],\bS^{\cI}[G],R)
\]
and applying this to the upper horizontal map in the diagram above, we get the homotopy cocartesian square in the proposition. 
\end{proof}

\begin{remark}
Using that the canonical map $G\to \Omega(BG)$ is an $\cI$-equivalence (see~\cite[Section~4]{Sagave-S_group-compl}), one may extend the result in Proposition~\ref{prop:Thom-suspension} to general based maps $\Sigma X\to BG$. We shall prove a space level version of this result in Section~\ref{sec:Thom-space-suspension}. 
\end{remark}

\section{Generalized Thom spectra from space level data}\label{sec:Thom-spectra-from-space-level-data}
In this section, we explain how the Thom spectrum functor $T^{\cI}$
from the previous section gives rise to an $R$-module Thom
spectrum functor taking space level data as input. For this purpose we first review the \emph{$\cI$-spacification functor} introduced in~\cite[Section 4.2]{Schlichtkrull_Thom-symmetric}.
\subsection{\texorpdfstring{$\cI$}{I}-spacification of space level data}\label{sec:I-spacification}
The $\cI$-spacification procedure works in general for a commutative $\cI$-space monoid $M$ and gives a multiplicative homotopy inverse of the homotopy colimit functor $(-)_{h\cI}\colon \Top^{\cI}/M\to \Top/M_{h\cI}$. Furthermore, if $\cD$ is an operad in $\Top$ that is augmented over the Barratt-Ecles operad $\cE$ (that is, equipped with a map $\cD\to \cE$), then the canonical $\cE$-action on $M_{h\cI}$ pulls back to a $\cD$-action, and it is proved in \cite[Corollary~6.9]{Schlichtkrull_Thom-symmetric} that 
$(-)_{h\cI}$ induces a functor $\Top^{\cI}[\cD]/M\to \Top[\cD]/M_{h\cI}$ relating the categories of $\cD$-algebras over $M$ and $M_{h\cI}$. The $\cI$-spacification functor is compatible with these actions and provides a homotopy inverse also in the algebra setting.

The $\cI$-spacification procedure is based on the bar resolution $\overline M$ of $M$ defined by 
\[ 
\overline{M}(\bld{n}) = \textstyle\hocolim_{(\cI\downarrow \bld{n})}
M\circ \pi_{\bld{n}}
\] 
in which $\pi_{\bld{n}}\colon (\cI\downarrow \bld{n})\to \cI$ denotes the forgetful functor.
Since each of the categories $(\cI\downarrow\bld{n})$ has a terminal object, the map from the
homotopy colimit to the colimit induces a levelwise equivalence
$t\colon \overline{M} \to M$. There is a canonical isomorphism $\colim_{\cI}\overline M\cong M_{h\cI}$ and we write $\pi\colon \overline M\to \const_{\cI}M_{h\cI}$ for the adjoint map of \mbox{$\cI$-spaces}. Thus, we have a diagram of 
$\cI$-equivalences
\[
\const_{\cI}M_{h\cI} \xl{\pi}\overline M\xr{t} M
\]
and it follows from the proof of~\cite[Lemma~6.7]{Schlichtkrull_Thom-symmetric} that this is a diagram of algebras over the Barratt-Eccles operad $\cE$. Consider the composite functor
\begin{equation}\label{eq:pi*-t-functors}
\Top/M_{h\cI}\xr{\pi^*} \Top^{\cI}/\overline M\xr{t} \Top^{\cI}/M
\end{equation}
where $\pi^*$ takes a space over $M_{h\cI}$, viewed as a constant $\cI$-space, to the pullback along $\pi$, and the functor $t$ is given by post-composition with the map $t$.

\begin{proposition}
The functors $\pi^*$ and $t$ define a chain of Quillen equivalences relating $\Top/M_{h\cI}$ equipped with the standard over-category model structure and $\Top^{\cI}/M$ equipped with the (absolute flat) over-category model structure.   
\end{proposition}
\begin{proof}
Recall from Remark~\ref{rem:proj-remark} that there is an (absolute) projective model structure on $\Top^{\cI}$ with the property that the colimit functor is a left Quillen functor. Let us write $(\Top^{\cI}/\overline M)_{\mathrm{proj}}$ for the category $\Top^{\cI}/\overline M$ equipped with the projective over-category model structure. Then we have a chain of Quillen adjunctions
\[
\xymatrix{
 \Top/M_{h\cI} \ar@<-.5ex>[r]_-{\pi^*} & (\Top^{\cI}/\overline{M})_{\textrm{proj}} \ar@<-.5ex>[l]_-{\mathrm{colim}} 
 \ar@<.5ex>[r]^-t & \Top^{\cI}/M \ar@<.5ex>[l]
 }
\]
which we claim to be Quillen equivalences. For the first adjunction we know that $\colim\colon \Top^{\cI}\to \Top$ is a Quillen equivalence and the claim then follows from the fact that $\pi$ is an $\cI$-equivalence. For the second adjunction we know that the identity functor defines a left Quillen functor from the projective to the flat model structure on $\Top^{\cI}/\overline M$ and that this is a Quillen equivalence. The result then follows since $t$ is an $\cI$-equivalence. 
\end{proof}

\begin{remark}
If $\cD$ is an operad augmented over the Barratt-Eccles operad, an analogous argument gives a chain of Quillen equivalences relating the categories $\Top[\cD]/M_{h\cI}$ and $\Top^{\cI}[\cD]/M$ equipped with the appropriate model structures.
\end{remark}

The functor $\pi^*$ in \eqref{eq:pi*-t-functors} is only homotopy invariant on fibrant objects and in order to remedy this we shall pre-compose with the standard Hurewicz fibrant replacement functor $\Gamma$ on $\Top/M_{h\cI}$. In detail, for a map $f\colon K\to M_{h\cI}$, let $\Gamma_f(K)$ be the space of pairs $(x,\omega)$ given by a point $x\in K$ and a path $\omega\colon I\to M_{h\cI}$ such that $\omega(0)=f(x)$. The functor $\Gamma$ then takes $f$ to the Hurewicz fibration $\Gamma(f)\colon\Gamma_f(K)\to M_{h\cI}$ mapping $(x,\omega)$ to $\omega(1)$. Putting these constructions together, we define the $\cI$-spacification functor 
\begin{equation}\label{eq:P_M}
P_M\colon \Top/M_{h\cI} \to \Top^{\cI}/M,\; (f\colon K \to M_{h\cI}) \mapsto (P_M(f)\colon P_{f}(K) \to M)
\end{equation} 
by letting $P_f(K)$ be the pullback in the diagram 
\[
\xymatrix@-1pc{
\const_{\cI}\Gamma_{f}(K) \ar@{->>}[d]^{\const_{\cI}\Gamma(f)} &&  P_{f}(K) \ar[ll] \ar@{->>}[d] \ar[drr]^-{P_M(f)} && \\
\const_{\cI}M_{h\cI} && \overline{M} \ar[ll]_-{\pi} \ar[rr]^(.4){t} && M. }
\]
Since the absolute flat model structure on $\Top^{\cI}$ is right proper, it is clear that 
$P_f(K)$ is $\cI$-equivalent to $\const_{\cI}K$. The effect of applying first $(-)_{h\cI}$ and then $P_M$ is dealt with in the next lemma.

\begin{lemma}\label{lem:P_M-after-hI}
Let $\alpha \colon X \to M$ be a map of $\cI$-spaces. Then $\alpha$ is $\cI$-equivalent to $P_M(\alpha_{h\cI})$ as objects in $\Top^{\cI}/M$.
\end{lemma}
\begin{proof}
Writing $\overline X$ for the bar resolution of $X$, the homotopy cartesian square of $\cI$-spaces 
\[\xymatrix@-1pc{
\const_{\cI}X_{h\cI} \ar[d] & \overline{X} \ar[l] \ar[d] \\
\const_{\cI}M_{h\cI} & \overline{M} \ar[l]
}\]
is $\cI$-equivalent to the homotopy pullback square 
defining $P_{\alpha_{h\cI}}(X_{h\cI})$. Using that $t\colon \overline{X} \to X$ is an $\cI$-equivalence, this proves the claim. 
\end{proof}
The category $\Top/M_{h\cI}$ inherits the structure of a monoidal category from the topological monoid structure of $M_{h\cI}$. Given objects $f\colon K\to M_{h\cI}$ and $g\colon L\to M_{h\cI}$, we write $f\times g$ for the monoidal product defined as the composition
\[
f\times g\colon K\times L\xr{f\times g} M_{h\cI}\times M_{h\cI}\to M_{h\cI}
\] 
in which the last map is the multiplication in $M_{h\cI}$. The monoidal unit is given by the unit $\iota\colon *\to M_{h\cI}$. 
Since $\pi$ and $t$ are maps of $\cI$-space monoids, $P_M$ canonically has the structure of a lax monoidal functor. Choosing cofibrant replacements of $P_f(K)$ and $P_g(L)$, we define the derived monoidal multiplication to be the composition
\begin{equation}\label{eq:P_M-derived-monoidal}
P_f(K)^{\mathrm{cof}}\boxtimes P_g(L)^{\mathrm{cof}}\to P_f(K)\boxtimes P_g(L)\to P_{f\times g}(K\times L)
\end{equation}
where the second map is the monoidal multiplication of $P_M$. 
\begin{lemma}\label{lem:P_M-monoidal}
The $\cI$-spacification functor $P_M$ is lax monoidal, and the monoidal unit $U^{\cI}\to P_{\iota}(*)$ and the derived monoidal multiplication  
\eqref{eq:P_M-derived-monoidal} are $\cI$-equivalences.
\end{lemma}
\begin{proof}
For $f$ and $g$ as above, the monoidal structure map is induced by the commutative diagram
\[
\xymatrix@-1pc{
\const_{\cI}\Gamma_f(K) \boxtimes \const_{\cI}\Gamma_g(L) \ar[r] \ar[d]& \const_{\cI}M_{h\cI}\boxtimes \const_{\cI}M_{h\cI}\ar[d] & \overline{M}\boxtimes \overline{M} \ar[l] \ar[d] \\
\const_{\cI}\Gamma_{f\times g}(K\times L) \ar[r] & \const_{\cI}M_{h\cI} & \overline{M}.\ar[l]
}
\] 
In order to show that the derived monoidal multiplication is an $\cI$-equivalence, 
we note that the left hand vertical map and the horizontal maps on the right hand side of the diagram induce weak homotopy equivalences after applying $(-)_{h\cI}$ (see \cite[Lemma~8.9]{Blumberg-C-S_THH-Thom}). We now use that $(-)_{h\cI}$ preserves and detects homotopy cartesian squares by~\cite[Corollary~11.4]{Sagave-S_diagram}, and that by~\cite[Lemma~2.25]{Sagave-S_group-compl}, the monoidal structure map of $(-)_{h\cI}$ is a weak homotopy equivalence when evaluated on cofibrant objects. This implies that the map in \eqref{eq:P_M-derived-monoidal} is $\cI$-equivalent to the map of horizontal homotopy pullbacks in the above diagram and is therefore an $\cI$-equivalence. 
\end{proof}

\subsection{The \texorpdfstring{$R$}{R}-module Thom spectrum functor on \texorpdfstring{$\Top/BG_{h\cI}$}{Top/BGhI}}\label{subsec:space-level-T}
Let $R$ be a (semistable as always) commutative symmetric ring spectrum, and let us assume for the rest of this section that the underlying symmetric spectrum of $R$ is flat. As usual, we let $G\to \GL_1^{\cI}(R)$ be a cofibrant replacement.

\begin{definition}
The $R$-module Thom spectrum functor $T$ on $\Top/BG_{h\cI}$ is the composition 
\begin{equation}\label{eq:space-level-Thom-sp-functor}
\Top/BG_{h\cI} \xrightarrow{P_{BG}} \Top^{\cI}/BG \xrightarrow{T^{\cI}} \Spsym{R}/\MGLoneIof{R}
\end{equation}
of the $\cI$-spacification functor for the commutative $\cI$-space monoid $BG$ and the 
  $R$-module Thom spectrum functor $T^{\cI}$ from Definition~\ref{def:TI}.
\end{definition}

The relation between the Thom spectrum functors $T$ and $T^{\cI}$ is recorded in the next proposition which is an immediate consequence of Lemma~\ref{lem:P_M-after-hI}.
\begin{proposition}\label{prop:TI-alpha-vs-TalphahI}
Let $\alpha \colon X \to BG$ be a map of $\cI$-spaces. Then $T^{\cI}(\alpha)$ and $T(\alpha_{h\cI})$ are naturally stably equivalent in $\Spsym{R}/\MGLoneIof{R}$.\qed
\end{proposition}

The basic properties of the Thom spectrum functor $T$ are summarized in the following proposition. We refer the reader to \cite{MMSS} and \cite[Section~7]{Sagave-S_diagram} for the notion of an \mbox{$h$-cofibration}. 

\begin{proposition}\label{prop:properties-of-T} 
The Thom spectrum functor $T$ has the following properties:
\begin{enumerate}[(i)]
\item It takes weak homotopy equivalences over $BG_{h\cI}$ to stable equivalences. 
\item It preserves colimits.
\item It preserves the tensor with an unbased space $Q$ in the sense that there is a natural isomorphism $T(f\times Q)\cong  T(f)\wedge Q_+$.
\item It takes maps over $BG_{h\cI}$ that are $h$-cofibrations in $\Top$ to $h$-cofibrations of $R$-modules.
\end{enumerate}
\end{proposition}

\begin{proof}
 It is clear from the construction that $P_{BG}$ takes weak equivalence over $BG_{h\cI}$ to level equivalences over
 $BG$, hence part (i) is a consequence of Proposition~\ref{prop:hty-inv-of-TI}. For part (ii) we observe that the functor $B(\bS^{\cI}[-],\bS^{\cI}[G],R)$ on $\Top^{\cI}_G$ preserves colimits, so that it remains to show that the same holds for the composite functor $U\circ P_{BG}$. Let us write $t^*EG\to \overline{BG}$ for the pullback of the universal fibration $EG\to BG$ along the levelwise equivalence $t\colon \overline{BG}\to BG$. Given a map $f\colon K\to BG_{h\cI}$, we can then identify $U(P_f(K))$ with the pullback of $f$ along the levelwise Hurewicz fibrant replacement of the composition 
 $t^*EG\to \overline{BG}\to BG_{h\cI}$, followed by the projection onto $EG$. Furthermore, the arguments in \cite{Lewis_when-cofibration} show that the space $BG_{h\cI}$ is locally equiconnected in the sense that the diagonal inclusion in $BG_{h\cI}\times BG_{h\cI}$ is an $h$-cofibration. Using this, the result follows from Propositions~1.1 and 1.2 in \cite[Section~IX]{LMS} which taken together state that pullback along a Hurewicz fibration with locally equiconnected codomain preserves colimits.  The claim in (iii) follows from the fact that both of the functors $P_{BG}$ and $T^{\cI}$ preserve tensors with unbased spaces. Finally, the proof of (iv) is analogous to the proof of the corresponding statement in 
\cite[Proposition~4.16]{Schlichtkrull_Thom-symmetric}: The functor $P_{BG}$ takes (not necessarily fiberwise) 
$h$-cofibrations over $BG_{h\cI}$ to fiberwise $h$-cofibrations over $BG$ by \cite[Proposition~IX1.11]{LMS}, and $T^{\cI}$ takes fiberwise $h$-cofibrations over $BG$ to $h$-cofibrations of $R$-modules. 
\end{proof}

Notice in particular, that by the above proposition $T$ takes homotopy cocartesian squares of spaces over 
$BG_{h\cI}$ to homotopy cocartesian squares of $R$-modules. 

Since $T^{\cI}$ is lax symmetric monoidal by Proposition~\ref{prop:TI-lax-monoidal} and  
$P_{BG}$ is lax monoidal by Lemma~\ref{lem:P_M-monoidal}, the composite functor $T$ is also lax monoidal. Given a pair of maps $f\colon K\to BG_{h\cI}$ and $g\colon L\to BG_{h\cI}$, we choose cofibrant replacements of the $R$-modules $T(f)$ and $T(g)$, and define the derived monoidal multiplication of $T$ to be the composite map
\begin{equation}\label{eq:T-derived-monoidal}
T(f)^{\mathrm{cof}}\wedge_R T(g)^{\mathrm{cof}}\to T(f)\wedge_R T(g)\to T(f\times g).
\end{equation}
Combining Proposition~\ref{prop:TI-derived-monoidal} and Lemma~\ref{lem:P_M-monoidal} we get the following result.
\begin{proposition}\label{prop:T-derived-monoidal}
The functor $T$ is lax monoidal, and the monoidal unit $R\to T(\iota)$ and the derived monoidal multiplication~\eqref{eq:T-derived-monoidal} are stable equivalences. \qed
\end{proposition}

Now let $\cD$ be an operad augmented over the Barratt-Eccles operad $\cE$, and let us view $BG_{h\cI}$ as a 
$\cD$-algebra by pulling back the canonical $\cE$-action along the augmentation. 

\begin{proposition}\label{prop:T-operad-version}
Let $\cD$ be an operad augmented over the Barratt-Eccles operad. Then the $R$-module Thom spectrum functor induces a functor 
\[
T\colon \Top[\cD]/BG_{h\cI}\to \Sp^{\Sigma}_{R}[\cD]/M\GL_1^{\cI}(R)
\] 
on the categories of $\cD$-algebras over $BG_{h\cI}$ and $M\GL_1^{\cI}(R)$. 
\end{proposition}
\begin{proof}
We know that the functor $T^{\cI}$ induces a functor on the corresponding categories of $\cD$-algebras by 
Corollary~\ref{cor:TI-preserves-operad-actions}. Arguing as in the proof of the corresponding statement in 
\cite[Proposition~6.8]{Schlichtkrull_Thom-symmetric}, one shows that the same holds for $P_{BG}$.
\end{proof}

In the case where $\cD$ is the associativity operad (the operad with $n$th space equal to $\Sigma_n$) the proposition says that the Thom spectrum functor takes topological monoids over $BG_{h\cI}$ to $R$-algebras over $M\GL_1^{\cI}(R)$.

\subsection{Thom spectra over space level suspensions}\label{sec:Thom-space-suspension}
Now we turn to the space level version of Proposition~\ref{prop:Thom-suspension} which we shall formulate for general maps of the form $f\colon \Sigma K\to BG_{h\cI}$. Given a map of based spaces $g\colon L\to G_{h\cI}$, we use the $\cI$-spacification functor from Section~\ref{sec:I-spacification} (with $M=G$) to pass to a map of based 
$\cI$-spaces
\[
P_G(g)\colon P_g(L)\to G.
\]
Proceeding as in Section~\ref{subsec:Thom-I-suspension}, this in turns extends to a map
\[
\Sigma P_G(g)\colon \Sigma P_g(L)\to BG.
\]

\begin{proposition}\label{prop:space-Thom-suspension}
To a well-based space $K$ and a based map $f\colon \Sigma K\to BG_{h\cI}$, there is functorially associated a diagram of well-based spaces of the form 
\[
K\xl{\simeq} \widehat K \xr{\hat f} G_{h\cI}
\]
such that $f$ and $\Sigma\hat f$ are weakly equivalent as objects in $\Top/BG_{h\cI}$, and there is a chain of natural stable equivalences
$
T(f)\simeq T^{\cI}(\Sigma P_G(\hat f)).
$
\end{proposition}
Combined with the description of $T^{\cI}(\Sigma P_G(\hat f))$ in Proposition~\ref{prop:Thom-suspension}, this gives a functorial description of $T(f)$ as an $R$-module which in turn leads to a cofiber sequence in the stable homotopy category of the form
\[
R\wedge K\to R\to T(f).
\]

The definition of the map $\Sigma\hat f$ figuring in  Proposition~\ref{prop:space-Thom-suspension} requires some explanation. Since we work with the homotopy colimit $BG_{h\cI}$ (i.e., $(BG)_{h\cI}$) as opposed to the weakly equivalent $B(G_{h\cI})$, we cannot identify the $1$-skeleton with the reduced suspension $\Sigma G_{h\cI}$ directly. Instead, we may view $BG_{h\cI}$ as the geometric realization of the simplicial space $[k]\mapsto (B_kG)_{h\cI}$ so that the 1-skeleton can be identified with the pushout of the diagram
\[
U^{\cI}_{h\cI} \ot U^{\cI}_{h\cI}\times \Delta^1\cup G_{h\cI}\times \partial \Delta^1 \to G_{h\cI}\times \Delta^1
\]
where $U^{\cI}_{h\cI}=B\cI$. Using that $\cI$ has an initial object, we can nonetheless define a natural embedding of $\Sigma G_{h\cI}$ in $BG_{h\cI}$: Let $h\colon U^{\cI}_{h\cI}\times I\to U^{\cI}_{h\cI}$ be the canonical null-homotopy with 
$h(-,0)$ the constant map and $h(-,1)$ the identity on $U^{\cI}_{h\cI}$, and let $p\colon G_{h\cI}\to U^{\cI}_{h\cI}$ be the projection. Representing $\Sigma G_{h\cI}$ as a quotient of $G_{h\cI}\times I$ in the usual way, we define 
\begin{equation}\label{eq:sigma-embedding}
\sigma\colon \Sigma G_{h\cI}\to BG_{h\cI},\quad
\sigma(x,t)=
\begin{cases}
h(p(x),3t), & \text{for $0\leq t\leq 1/3$}\\
(x,(3t-1,2-3t)), &\text{for $1/3\leq t\leq 2/3$}\\
h(p(x),3-3t), &\text{for $2/3\leq t\leq 1$.}
\end{cases}
\end{equation} 
Given a map of based spaces $g\colon L\to G_{h\cI}$, we write $\Sigma g$ both for the induced map of suspensions $\Sigma L\to \Sigma G_{h\cI}$ and for the composition
\[
\Sigma g\colon \Sigma L\to \Sigma G_{h\cI} \xr{\sigma} BG_{h\cI}.
\]
The context will make the meaning clear. In the next lemma we compare the effect of suspending before or after passing to $\cI$-spaces.

\begin{lemma}\label{lem:Rect-Sigma}
Given a well-based space $L$ and a based map $g\colon L\to G_{h\cI}$, there is a chain of natural $\cI$-equivalences relating $P_{BG}(\Sigma g)$ and $\Sigma P_G(g)$ as objects in $\Top^{\cI}/BG$. 
\end{lemma}
\begin{proof}
We first observe that the map $\sigma$ in \eqref{eq:sigma-embedding} admits a canonical lift to a map of bar resolutions $\sigma\colon \Sigma\overline G\to \overline{BG}$, such that the left hand square in the diagram 
\[
\xymatrix@-.5pc{
\const_{\cI}\Sigma G_{h\cI} \ar[d]_{\sigma}& \Sigma \overline G \ar[l]_-{\Sigma\pi} \ar[r]^-{\Sigma t} \ar[d]_{\sigma} & \Sigma G \ar[d]\\
\const_{\cI}BG_{h\cI} & \overline{BG} \ar[l]_-{\pi} \ar[r]^-t & BG
}
\]
is strictly commutative and the right hand square 
is homotopy commutative by a canonical homotopy $H\colon \Sigma\overline G \times I\to BG$ starting at $t\circ\sigma$ and ending at $\Sigma t$.
Putting all this together, we get the commutative diagram
\[
\xymatrix@C-1pc{
\const_{\cI}\Sigma L \ar[r]^-{\simeq} \ar[ddr]_-{\Sigma g}& \const_{\cI}\Sigma\Gamma_g(L) \ar[d]^{\Sigma\Gamma(g)} & 
\Sigma P_g(L) \ar[l]_-{\simeq} \ar[r]^-{i_0} \ar[d] & \Sigma P_g(L)\times I \ar[d]
& \Sigma P_g(L) \ar[l]_-{i_1} \ar[ddl]^{\Sigma P_G(g)}\\
& \const_{\cI}\Sigma G_{h\cI} \ar[d]^{\sigma} & \Sigma \overline G \ar[l]_-{\Sigma\pi} \ar[r]^-{i_0} \ar[d]^{\sigma} & 
\Sigma \overline G\times I \ar[d]_{H}& \\
& \const_{\cI}BG_{h\cI} & \ar[r] \overline{BG} \ar[l]_-{\pi}\ar[r]^-t & BG. &
}
\]
Here the left part of the diagram gives a chain of $\cI$-equivalences 
\[
P_{\Sigma g}(\Sigma L)\xr{\simeq} P_{\sigma\circ\Sigma\Gamma(g)}(\Sigma\Gamma_g(L)) 
\xl{\simeq} \Sigma P_g(L)
\]
in $\Top^{\cI}/BG$, but where $\Sigma P_g(K)$ maps to $BG$ via the composition indicated in the middle of the diagram. Composing with the maps in the right part of the diagram, we get the chain of $\cI$-equivalences in the lemma.
\end{proof}

\begin{proof}[Proof of Proposition~\ref{prop:space-Thom-suspension}]
Let $\check f\colon K\to \Omega(BG_{h\cI})$ be the adjoint of $f$, and let $\hat f$ be defined as the pullback indicated in the diagram
\[
\xymatrix@-.5pc{
& \widehat K\ar[r]^{\hat f} \ar[d]_{\simeq}& G_{h\cI} \ar[d]^{\check\sigma}_{\simeq}\\
K \ar@<.5ex>[r]^-{\simeq} & \ar@<.5ex>[l] \Gamma_{\check f}(K) \ar[r]^-{\Gamma(\check f)} & \Omega(BG_{h\cI}).
}
\]
Then we have a chain of weak equivalences relating $f$ and $\Sigma \hat f\colon \Sigma \widehat K\to BG_{h\cI}$ as objects in $\Top/BG_{h\cI}$ which gives a chain of stable equivalences $T(f)\simeq T(\Sigma \hat f)$. Now the conclusion follows from Lemma~\ref{lem:Rect-Sigma} and the homotopy invariance of the Thom spectrum functor $T^{\cI}$.
\end{proof}

\section{Quotient spectra as Thom spectra associated to \texorpdfstring{$SU(n)$}{SU(n)}}\label{sec:Quotient-Thom}
In this section we specialize to the case where $R$ is a commutative symmetric ring spectrum that is \emph{even} in the sense that $\pi_*(R)$ is concentrated in even degrees. As usual $R$ is supposed to be semistable and in this section we also assume that the underlying symmetric spectrum of $R$ is flat. We again write $G\to \GL_1^{\cI}(R)$ for a cofibrant replacement. 

Let $SU(n)$ denote the special unitary group. Our main concern is to analyze Thom spectra associated to loop maps of the form $SU(n)\to BG_{h\cI}$. Since we are mainly interested in Thom spectra that are strict symmetric ring spectra (as opposed to a more relaxed notion of $A_{\infty}$ ring spectra), we shall model such loop maps by maps of actual topological monoids. For this reason it will be convenient to work with the model $B_1(-)$ of the classifying space functor for topological monoids introduced by Fiedorowicz \cite{Fiedorowicz_classifying}. We review the relevant details of this construction in Appendix~\ref{app:loop-rectification}. Writing 
$\cA\Top/BG_{h\cI}$ for the category of topological monoids over $BG_{h\cI}$, we shall use the ``loop functor'' from 
Section~\ref{subsec:rectified-loop}, 
\[
\Omega'\colon \Top_*/B_1BG_{h\cI}\to \cA\Top/BG_{h\cI},\quad (K\xr{g} B_1BG_{h\cI})\mapsto (\Omega'_g(K)\xr{\Omega'(g)}BG_{h\cI}),
\]  
which models the looping of a based map $g\colon K\to B_1 BG_{h\cI}$ by a map of actual topological monoids $\Omega'(g)\colon \Omega'_g(K)\to BG_{h\cI}$. 

Now suppose we are given a based map $g\colon B_1SU(n)\to B_1BG_{h\cI}$. Then we introduce the notation $SU'(n)$ for 
$\Omega_g'(B_1(SU(n)))$, so that the associated loop map is realized by the map of topological monoids $\Omega'(g)\colon SU'(n)\to BG_{h\cI}$. For each $m=1,\dots,n$,  we write $SU'(m)=\Omega'_{g_m}(B_1SU(m))$, where $g_m\colon B_1SU(m)\to BG_{h\cI}$ is the ``restriction'' of $g$ obtained by precomposing with the map of classifying spaces induced by the standard inclusion of $SU(m)$ in $SU(n)$. The notation is justified by the fact that there is a canonical chain of weak equivalences of topological monoids relating $SU'(m)$ and $SU(m)$, and it follows from Proposition~\ref{prop:Omega'-compatible} that these equivalences are compatible when $m$ varies. 
Until further notice we fix a based map $g$ as above and simplify the notation by writing 
\[
T(SU(m))=T(\Omega'(g_m))
\] 
for $1\leq m\leq n$. Thus, the $R$-algebra Thom spectrum $T(SU(m))$ depends on the map $g$ even though this is not visible in the notation.

\subsection{The structure of \texorpdfstring{$T(SU(n))$}{T(SU(n))}}\label{sec:T(SU(n))-structure}
Our assumption that $R$ be even implies that $R$ is complex orientable (see \cite{Adams_Stable}), so we may choose an element $x\in \tilde R^2(\bC P^{n-1})$ such that the map of $\pi_*(R)$-algebras $\pi_*(R)[x]/x^{n}\to R^*(\bC P^{n-1})$ is an isomorphism. Next recall (see e.g.\ \S10 of Chapter IV in \cite{Whitehead-Elements}) that the reduced suspension  $\Sigma \bC P^{n-1}$  admits a canonical embedding in $SU(n)$. Composing with $\Omega'(g)$ and the chain of equivalences relating 
$SU(n)$ and $SU'(n)$, we get a well-defined homotopy class
\begin{equation}\label{eq:SigmaCP-class}
\Sigma\bC P^{n-1}\to SU(n)\simeq SU'(n) \xr{\Omega'(g)}BG_{h\cI}
\end{equation}
whose adjoint determines a cohomology class
\[
u\in [\bC P^{n-1}, \Omega(BG_{h\cI})]_*\cong  [\bC P^{n-1},G_{h\cI}]_*\subset [\bC P^{n-1},G_{h\cI}]\cong R^0(\bC P^{n-1})^\times.
\]
For the last isomorphism we use that the topological monoid $G_{h\cI}$ models the units of $R$, hence represents the functor taking a space $X$ to the units $R^0(X)^{\times}$. It follows that we can write $u$ uniquely in the form
\begin{equation}\label{eq:u-representation}
u=1+u_1x+u_2x^2+\dots+u_{n-1}x^{n-1},\qquad u_i\in \pi_{2i}(R).
\end{equation}
For $c$ a natural number, let $F^{\bS}_c\colon \Top_*\to\Sp^{\Sigma}$ be the level $c$ free symmetric spectrum functor, see \cites{HSS, MMSS}. 
Since $R$ is assumed to be semistable, we can represent $u_i$ by an actual map of symmetric spectra $u_i\colon F^{\bS}_c(S^{2i+c})\to R$ for a suitable constant $c$. In the proposition below we use the notation $\Sigma^{2m}(-)$ for the smash product with $F^{\bS}_c(S^{2m+c})$.

\begin{proposition}\label{prop:T(SU)-cofiber}
There are homotopy cofiber sequences of $R$-modules
\[
\Sigma^{2m}T(SU(m))\xr{u_m} T(SU(m))\xr{} T(SU(m+1)).
\]
for $1\leq m<n$.
\end{proposition}
Here the first map is given by $u_m$ and the $R$-module structure on $T(SU(m))$. The precise meaning of the term ``homotopy cofiber sequence'' is as follows: There is a chain of stable equivalences in the category of $R$-modules under $T(SU(m))$ relating $T(SU(m+1))$ to the mapping cone of the first map. 

The proof of Proposition~\ref{prop:T(SU)-cofiber} is based on two lemmas. In the first lemma we consider the analogues of the homotopy classes \eqref{eq:SigmaCP-class} for $1\leq m\leq n$. We arrange to have a set of representatives $\Sigma\bC P^{m-1}\to SU'(m)$ that are compatible when $m$ varies, and we write $T(\Sigma \bC P^{m-1})$ for the associated Thom spectra. Let  $T(SU(m))^{\mathrm{cof}}$ be a cofibrant replacement of the $R$-module $T(SU(m))$.

\begin{lemma}\label{lem:T(SU)-square}
For $1\leq m<n$, the commutative square of $R$-modules
\[
\xymatrix@-1pc{
T(\Sigma\bC P^{m-1})\wedge_RT(SU(m))^{\mathrm{cof}} \ar[r] \ar[d]&T(SU(m))\ar[d] \\
T(\Sigma\bC P^{m})\wedge_RT(SU(m))^{\mathrm{cof}} \ar[r] &T(SU(m+1))
}
\]
is homotopy cocartesian.
\end{lemma}

\begin{proof}
Using the standard inclusion of $SU(m)$ in $SU(m+1)$ and the multiplicative structures of these groups, we get the commutative diagram
\begin{equation}\label{eq:SU-pushout} 
\xymatrix@-1pc{
\Sigma\bC P^{m-1}\times SU(m) \ar[r] \ar[d] & SU(m)\times SU(m) \ar[r]\ar[d]  & SU(m)\ar[d]\\
\Sigma\bC P^{m}\times SU(m) \ar[r] & SU(m+1)\times SU(m) \ar[r] & SU(m+1)
}
\end{equation}
in which the outer diagram is a pushout square as follows from the statements listed as (i) and (ii) on page 345 in 
\cite{Whitehead-Elements}. Since the chains of weak equivalences $SU(m)\simeq SU'(m)$ are multiplicative and compatible when $m$ varies, it follows that the analogous square with $SU'(m)$ and $SU'(m+1)$ is homotopy cocartesian. The latter may be viewed as a diagram in $\Top/BG_{h\cI}$, so applying the Thom spectrum functor gives a homotopy cocartesian square of $R$-modules. Hence the result follows from Proposition~\ref{prop:T-derived-monoidal} (it suffices to cofibrantly replace one of the factors).
\end{proof}

In the second lemma we analyze the maps $T(\Sigma \bC P^{m-1})\to T(\Sigma \bC P^{m})$. Recall from \cite{Adams_Stable} that the $R$-homology of $\bC P^{n-1}$ is given by 
\[
R_*(\bC P^{n-1})=\pi_*(R)\{\beta_0,\dots,\beta_{n-1} \} 
\]
where $\beta_i\in R_{2i}(\bC P^{n-1})$ is dual to $x^i$. The classes $\beta_i$ may be realized as maps of symmetric spectra 
$\beta_i\colon F_c^{\bS}(S^{2i+c})\to R\wedge \bC P^{n-1}_+$, for a suitable constant $c$, and give rise to a stable equivalence of $R$-modules
\begin{equation}\label{eq:RCP-equivalence}
\textstyle\bigvee_{i=0}^{n-1} \Sigma^{2i}R= \bigvee_{i=0}^{n-1} 
F_c^{\bS}(S^{2i+c})\wedge R \xr{\vee \beta_i} R\wedge \bC P^{n-1}_+.
\end{equation}

\begin{lemma}\label{lem:T-Sigma-CP-sequence}
For $1\leq m<n$ there are homotopy cofiber sequences of $R$-modules
\[
\Sigma^{2m}R\xr{u_m} T(\Sigma\bC P^{m-1}) \xr{} T(\Sigma \bC P^m).
\]
\end{lemma}
Here we again write $\Sigma^{2m}(-)$ for the smash product with $F_c^{\bS}(S^{2m+c})$ and the first map is defined by the composition
\[
F_c^{\bS}(S^{2m+c})\wedge R \xr{u_m\wedge R} R\wedge R\xr{} R \xr{} T(\iota) \xr{} T(\Sigma \bC P^{m-1}),
\]
where the maps are given respectively by $u_m$, the multiplication in $R$, the monoidal unit of $T$, and the inclusion of the base point in $\Sigma \bC P^{n-1}$. The meaning of the term ``homotopy cofiber sequence'' is as in Proposition~\ref{prop:T(SU)-cofiber}.
\begin{proof}[Proof of Lemma~\ref{lem:T-Sigma-CP-sequence}]
Let us write $f_m\colon \Sigma\bC P^m\to BG_{h\cI}$ for the based maps giving rise to the Thom spectra $T(\Sigma\bC P^m)$. 
Applying Proposition~\ref{prop:space-Thom-suspension} to the maps $f_m$ (with $K_m=\bC P^m$) we get a sequence of maps $\hat f_m\colon \widehat \bC P^m\to G_{h\cI}$ such that $\widehat \bC P^m\simeq \bC P^m$ and there are stable equivalences $T(\Sigma\bC P^m)\simeq T^{\cI}(\Sigma P_G(\hat f_m))$. Furthermore, there are induced maps $\widehat \bC P^m\to \widehat \bC P^{m+1}$ such that these equivalences are compatible with the inclusion of $\bC P^m$ in $\bC P^{m+1}$. The symmetric spectrum $\bS^{\cI}[P_{\hat f_m}(\widehat \bC P^m)]$ is a (semistable) model of $\Sigma^{\infty}(\bC P^m_+)$, so we get a stable equivalence of $R$-modules
\[
\vee \beta_i\colon\textstyle\bigvee_{i=0}^{m} 
 F_c^{\bS}(S^{2i+c})\wedge R \xr{\simeq} \bS^{\cI}[P_{\hat f_m}(\widehat \bC P^m)]\wedge R
\]
as in \eqref{eq:RCP-equivalence}. Proceeding inductively, we may assume these equivalences to be compatible with the maps $\widehat \bC P^m\to \widehat \bC P^{m+1}$. Hence it follows from the description in Proposition~\ref{prop:Thom-suspension} that there are homotopy cocartesian squares of the form
\[
\xymatrix@-.5pc{
\bigvee_{i=0}^{m} F_c^{\bS}(S^{2i+c})\wedge R \ar[r] \ar[d]_{\vee u_i}& \bigvee_{i=0}^{m}  F_c^{\bS}(CS^{2i+c})\wedge R\ar[d]\\
B(\bS^{\cI}[G],\bS^{\cI}[G],R) \ar[r] & T^{\cI}(\Sigma P_G(\hat f_m))
}
\]
where the notation indicates that the maps $u_i$ represent the corresponding homotopy classes in \eqref{eq:u-representation} under the canonical stable equivalence $R\simeq B(\bS^{\cI}[G],\bS^{\cI}[G],R)$. Furthermore, we may arrange for these squares to be compatible when $m$ varies. This gives the statement in the lemma.  
\end{proof}

\begin{proof}[Proof of Proposition~\ref{prop:T(SU)-cofiber}]
Smashing the homotopy cofiber sequence in Lemma~\ref{lem:T-Sigma-CP-sequence} with the cofibrant $R$-module 
$T(SU(m))^{\mathrm{cof}}$, we get the homotopy cofiber sequence in the left column of the diagram
\[
\xymatrix@-1pc{
\Sigma^{2m}R\wedge_R T(SU(m))^{\mathrm{cof}} \ar[r]^-{\simeq} \ar[d]& \Sigma^{2m}T(SU(m))\ar[d]\\
T(\Sigma\bC P^{m-1})\wedge_RT(SU(m))^{\mathrm{cof}} \ar[r] \ar[d]&  T(SU(m))  \ar[d] \\
T(\Sigma\bC P^{m})\wedge_RT(SU(m))^{\mathrm{cof}}  \ar[r]  & T(SU(m+1)).
}
\]
The result follows since the bottom square is homotopy cocartesian by Lemma~\ref{lem:T(SU)-square}.
\end{proof}

\subsection{Quotient spectra as generalized Thom spectra}\label{subsec:quotient-Thom}
We recall some facts about quotient constructions for symmetric spectra. Let as usual $R$ be a semistable commutative symmetric ring spectra and let $x\in \pi_d(R)$ be a homotopy class represented by a map $f\colon S^{d_2}\to R_{d_1}$ with $d_2-d_1=d$. The map $f$ extends to a map of symmetric spectra $F_{d_1}^{\bS}(S^{d_2})\to R$ and using the multiplication in $R$ we get a map of $R$-modules $f\colon F^{\bS}_{d_1}(S^{d_2})\wedge R\to R\wedge R\to R$ that represents multiplication by $x$ on homotopy groups. We define $R/f$ to be the mapping cone of this map so that we have a pushout diagram
\[
\xymatrix@-.5pc
{
 F_{d_1}^{\bS}(S^{d_2})\wedge R \ar[r] \ar[d]_f&  F_{d_1}^{\bS}(CS^{d_2})\wedge R\ar[d]\\
 R \ar[r] & R/f
}
\]
where $CS^d$ denotes the reduced cone on $S^d$. Notice that, with this definition, $R/f$ is a flat $R$-module. It is clear that if $g\colon S^{d_2}\to R_{d_1}$ represents the same class in $\pi_{d_2}(R_{d_1})$, then $M/f$ and $M/g$ are stably equivalent $R$-modules. Furthermore, if $\sigma(f)\colon S^{d_2+1}\to R_{d_1+1}$ denotes the ``suspension'' of $f$, then  we have  a commutative diagram 
\[
\xymatrix@-1pc{
F_{d_1+1}^{\bS}(S^{d_2+1}) \ar[rr]^-{\simeq} \ar[dr]_{\sigma(f)} & & F_{d_1}^{\bS}(S^{d_2}) \ar[dl]^f \\
& R & 
}
\] 
where the horizontal arrow is the canonical stable equivalence. From this we get a stable equivalence $R/\sigma(f)\xr{\sim} R/f$ which shows that the homotopy type of the $R$-module $R/f$ only depends on the class $x\in \pi_d(R)$ represented by $f$. We use the notation $R/x$ for this homotopy type. Given a sequence of classes $x_1,\dots,x_n$ in $\pi_*(R)$, we chose representatives $f_1,\dots,f_n$ as above and define $R/(x_1,\dots,x_n)$ to be the homotopy type of $R$-modules represented by the smash product \[
R/(f_1,\dots,f_n)=(R/f_1)\wedge_R\dots\wedge_R(R/f_n). 
\]
The homotopy type is well-defined since the $R$-modules $R/f_i$ are flat and it follows from the definition that there are homotopy cofiber sequences
\[
\Sigma^{|x_i|}R/(x_1,\dots,x_{i-1})\xr{x_i} R/(x_1,\dots,x_{i-1}) \to R/(x_1,\dots,x_{i}).
\]
If $x_1,\dots,x_n$ is a regular sequence in $\pi_*(R)$, then these homotopy cofiber sequences split up into short exact sequences of homotopy groups from which we deduce an isomorphism
\[
\pi_*(R)/(x_1,\dots,x_n)\xr{\simeq} \pi_*\big(R/(x_1,\dots,x_n)\big).
\]
The corresponding homomorphism may well fail to be an isomorphism if the sequence is not regular.

Now we return to the setting from the beginning of this section and consider $R$-algebra Thom spectra $T(SU(n))$ associated to loop maps as defined there. Our main result is stated in the next theorem and shows that the process described in the previous subsection can be reversed so that instead of starting with a loop map we start with a sequence of homotopy classes  in 
$\pi_*(R)$.

\begin{theorem}\label{thm:T(SU(n))-quotient}
Given homotopy classes $u_i\in \pi_{2i}(R)$ for $i=1,\dots,n-1$, there exists a based map $B_1SU(n)\to B_1BG_{h\cI}$ such that the homotopy type of the $R$-module underlying the $R$-algebra Thom spectrum $T(SU(n))$ of the associated loop map is described by
\[
T(SU(n))\simeq R/(u_1,\dots,u_{n-1}).
\]
\end{theorem}
\begin{proof}
Working in the homotopy category, the classes $u_i$ (and the chosen orientation class $x$) determine a based map $u\colon \Sigma\bC P^{n-1}\to BG_{h\cI}$ as explained in Section~\ref{sec:T(SU(n))-structure}. After suspending once we consider the extension problem of filling in the map $g$ in the diagram
\[
\xymatrix@-1pc{
\Sigma(\Sigma\bC P^{n-1}) \ar[r]^-{\Sigma u} \ar[d]& \Sigma BG_{h\cI} \ar[r] & B_1BG_{h\cI}\\
\Sigma SU(n) \ar[d]& &\\
B_1SU(n). \ar@{-->}[uurr]_g& & 
}
\]
The obstructions to the extension problem lie in the cohomology groups 
\[
\operatorname{H}^{k+1}\big((B_1SU(n),\Sigma^2 \bC P^{n-1}),\pi_k(B_1BG_{h\cI}) \big)
\]
which are trivial since the groups $\pi_k(B_1BG_{h\cI})\cong \pi_{k-2}(\GL_1^{\cI}(R)_{h\cI})$ are concentrated in even degrees. If $g$ is such an extension, then $u$ can be recovered as the composition
\[
u\colon \Sigma\bC P^{n-1} \to SU(n)\simeq SU'(n) \xr{\Omega'(g)}BG_{h\cI}.
\]
We claim that the $R$-algebra Thom spectrum $T(SU(n))$ associated to $\Omega'(g)$ has the homotopy type described in the theorem. Proceeding as in Section~\ref{sec:T(SU(n))-structure}, we represent the classes $u_i$ by actual maps of symmetric spectra 
$u_i\colon F_c^{\bS}(S^{2i+c})\to R$ so as to get homotopy cofiber sequences of $R$-module spectra
\[
F_c^{\bS}(S^{2m+c})\wedge R/(u_1,\dots,u_{m-1})\xr{u_m} R/(u_1,\dots,u_{m-1}) \to R/(u_1,\dots,u_{m}).
\]
Comparing these homotopy cofiber sequences with those of Proposition~\ref{prop:T(SU)-cofiber}, we conclude by induction that the $R$-module $T(SU(m))$ is equivalent to $R/(u_1,\dots,u_{m-1})$ for all $m=1,\dots,n$. 
\end{proof}

Recall that a commutative symmetric ring spectrum $R$ is said to be $2$-periodic if there exists an element $w\in \pi_2(R)$ which is a unit in the graded ring $\pi_*(R)$. The existence of such a unit allows us to strengthen the statement in the above theorem.

\begin{corollary}\label{cor:T(SU(n))-quotient-2-periodic}
If the commutative symmetric ring spectrum $R$ in Theorem~\ref{thm:T(SU(n))-quotient} is $2$-periodic (as well as even), then the statement of the theorem holds for any family $u_1,\dots, u_{n-1}$ of even dimensional homotopy classes. (That is, the requirement that $u_i$ is an element in $\pi_{2i}(R)$ is no longer needed). 
\end{corollary}
\begin{proof}
Let $w\in \pi_2(R)$ be a unit in the graded ring $\pi_*(R)$. Given a sequence of homotopy classes $u_1,\dots,u_{n-1}$ in $\pi_*(R)$, we set $v_i=u_i\cdot w^{i-|u_i|/2}$ so that the classes $v_1,\dots,v_{n-1}$ satisfy the degree requirements in 
Theorem~\ref{thm:T(SU(n))-quotient}. Let $f_i\colon S^{d_{i2}}\to R_{d_{i1}}$ be a representative for $u_i$ and let 
$g_i\colon S^{e_{i2}}\to R_{e_{i1}}$ be a representative for $w^{i-|u_i|/2}$. Then the map $f_i\cdot g_i$ defined by
\[
f_i\cdot g_i\colon S^{d_{i2}}\wedge S^{e_{i2}} \to R_{d_{i1}}\wedge R_{e_{i1}} \to R_{d_{i1}+e_{i1}}
\]
is a representative for $v_i$ (see \cite[Section~4]{Sagave-S_diagram} and the discussion in \cite{Schwede_SymSp}). Consider the commutative diagram
\[
\xymatrix@-.5pc{
R \ar@{=}[d] & \ar[l]_-{f_i\cdot g_i}  F^{\bS}_{d_{i1}}(S^{d_{i2}})\wedge F^{\bS}_{e_{i1}}(S^{e_{i2}})\wedge R \ar[r] \ar[d]^{\id\wedge g_i}  & 
 F^{\bS}_{d_{i1}}(CS^{d_{i2}})\wedge F^{\bS}_{e_{i1}}(S^{e_{i2}})\wedge R \ar[d]^{\id\wedge g_i}  \\
R & \ar[l]_-{f_i}  F^{\bS}_{d_{i1}}(S^{d_{i2}})\wedge R \ar[r] & 
 F^{\bS}_{d_{i1}}(CS^{d_{i2}})\wedge R 
}
\]
where the maps are induced by $f_i$, $g_i$, and $f_i\cdot g_i$ as explained in the beginning of this subsection. Here we use the canonical identification of $F^{\bS}_{d_{i1}}(S^{d_{i2}})\wedge F^{\bS}_{e_{i1}}(S^{e_{i2}})$ with $F^{\bS}_{d_{i1}+e_{i1}}(S^{d_{i2}}\wedge S^{e_{i2}})$. The vertical maps induced by $g_i$ are stable equivalences since $g_i$ represents a unit in $\pi_*(R)$. Evaluating the horizontal pushouts, the diagram thus gives us a stable equivalence $R/(f_i\cdot g_i)\xr{\sim} R/f_i$. Combining this with Theorem~\ref{thm:T(SU(n))-quotient} we get the stable equivalences
\[
T(SU(n))\simeq  (R/f_1\cdot g_1)\wedge_R\dots\wedge_R(R/f_{n-1}\cdot g_{n-1})
\simeq (R/f_1)\wedge_R\dots\wedge_R(R/f_{n-1}) 
\]
where the last term represents the homotopy type $R/(u_1,\dots,u_{n-1})$.
\end{proof}


\section{Topological Hochschild homology of Thom spectra}\label{sec:THH-of-Thom}
Consider in general a monoid $A$ in a symmetric monoidal category $(\cA,\Box,1_{\cA})$.
Then the \emph{cyclic bar construction} $B^{\cy}_{\bullet}(A)$ is the simplicial object
$[k] \longmapsto A^{\Box(k+1)}$ with simplicial structure maps defined as for the standard Hochschild complex of an algebra (see e.g.~\cite[Section~1]{Blumberg-C-S_THH-Thom} for more details). We shall use the notation  $B^{\cy}(A)$ for the geometric realization of $B^{\cy}_{\bullet}(A)$ when this makes sense in the category $\cA$. In the case of a commutative symmetric ring spectrum $R$ and the symmetric monoidal category of modules $\Sp^{\Sigma}_R$, a monoid $A$ is the same thing as an $R$-algebra and the cyclic bar construction takes the form
\begin{equation}\label{eq:R-algebra-cyclic-bar}
B^{\cy}_{\bullet}(A)\colon [k]\mapsto A\wedge_R\dots \wedge_R A \quad \text{(with $k+1$ smash factors)}. 
\end{equation}
It is well-known that the topological Hochschild homology of an $R$-algebra can be modelled by the cyclic bar construction under suitable cofibrancy conditions. The conditions we impose ensure that our definition is equivalent to the standard definition of topological Hochschild homology as a derived smash product, see \cite[Chapter~IX]{EKMM} and \cite[Section~4]{Shipley_THH}. 

\begin{definition}
Let $R$ be a commutative symmetric ring spectrum and let $A$ be an $R$-algebra such that the unit $R\to A$ is an $h$-cofibration and the underlying $R$-module of $A$ is flat. Then the topological Hochschild homology $\THH^R(A)$ of $A$ is the geometric realization $B^{\cy}(A)$ of the cyclic bar construction in \eqref{eq:R-algebra-cyclic-bar}.
\end{definition}

This definition is homotopy invariant: If $A$ and $B$ are $R$-algebras that satisfy the conditions in the definition and $A\to B$ is a stable equivalence, then the induced map $\THH^R(A)\to \THH^R(B)$ is also a stable equivalence. Notice that the conditions on $A$ hold if $A$ is cofibrant in the category of $R$-algebras $R/\cA\Sp^{\Sigma}$ equipped with the absolute or positive flat model structures. (Recall that $\cA\Sp^{\Sigma}$ denotes the category of symmetric ring spectra.) They also hold if $A$ is commutative and cofibrant in the corresponding positive flat model structure on 
$R/\cC\Sp^{\Sigma}$.

\begin{remark}\label{rem:THH-cof-replacement}
When we want to consider the topological Hochschild homology of an $R$-algebra $A$ that does not satisfy the conditions in the definition, we should first find a suitable replacement in the form of a stable equivalence $A^c\xr{\sim} A$ where $A^c$ is an 
$R$-algebra that does satisfy the conditions. We may then take $B^{\cy}(A^c)$ as our definition of $\THH^R(A)$.
\end{remark} 

Returning to the discussion of Thom spectra, we assume for the rest of this section that the commutative symmetric ring spectrum $R$ is semistable and that the underlying symmetric spectrum of $R$ is flat. Recall the $R$-module Thom spectrum functor $T^{\cI}$ on $\Top^{\cI}/BG$ introduced in Section~\ref{subsec:Thom-on-Top^I/BG}.

\begin{lemma}\label{lem:Thom-realization}
The Thom spectrum functor $T^{\cI}$ preserves geometric realization of simplicial objects.
\end{lemma}
\begin{proof}
By definition, a simplicial object in  $\Top^{\cI}/BG$ amounts to a map of simplicial $\cI$-spaces $\alpha_{\bullet}\colon X_{\bullet}\to BG$ where we view $BG$ as a constant simplicial $\cI$-space. Writing $\alpha\colon X\to BG$ for the geometric realization of $\alpha_{\bullet}$, the statement in the lemma says that there is a natural isomorphism of $R$-modules $|T^{\cI}(\alpha_{\bullet})|\cong T^{\cI}(\alpha)$ over $M\GL_1^{\cI}(R)$. Consider the three functors in Definition~\ref{def:TI}. The first functor $U$ preserves geometric realization since it is given by a pullback construction and geometric realization of simplicial spaces commutes with pullbacks. The remaining two functors preserve geometric realization since they preserve colimits and tensors with spaces. This gives the result. 
\end{proof}

Given a map of $\cI$-space monoids $\alpha\colon M\to BG$, we use the notation $B^{\cy}(\alpha)$ for the composite map 
\[
B^{\cy}(\alpha)\colon B^{\cy}(M)\to B^{\cy}(BG)\to BG
\]
where the last map is induced by the multiplication in the commutative $\cI$-space monoid $BG$. With this definition, $B^{\cy}(\alpha)$ is the cyclic bar construction internal to the symmetric monoidal category $\Top^{\cI}/BG$. In the following proposition we also consider the cyclic bar construction internal to $\Sp^{\Sigma}_R$ and apply it to a cofibrant replacement $T^{\cI}(\alpha)^{\cof}$ of $T^{\cI}(\alpha)$ so that $B^{\cy}(T^{\cI}(\alpha)^{\cof})$ is a model of $\THH^R(T^{\cI}(\alpha))$, cf.\ Remark~\ref{rem:THH-cof-replacement}. We say that an $\cI$-space monoid $M$ is \emph{well-based} if the inclusion of the unit $U^{\cI}\to M$ is an $h$-cofibration (in the sense of~\cite[Section~7]{Sagave-S_diagram}). 

\begin{proposition}\label{prop:Bcy-vs-T}
Let $\alpha\colon M\to BG$ be a map of $\cI$-space monoids and suppose that $M$ is well-based and that the underlying $\cI$-space of $M$ is flat. 
\begin{enumerate}[(i)]
\item
If $T^{\cI}(\alpha)^{\cof}\to T^{\cI}(\alpha)$ is a cofibrant replacement in the category of $R$-algebras, then there is a stable equivalence of $R$-modules $B^{\cy}(T^{\cI}(\alpha)^{\cof})\xr{\sim}T^{\cI}(B^{\cy}(\alpha))$.

\item
If $M$ is commutative and $T^{\cI}(\alpha)^{\cof}\to T^{\cI}(\alpha)$ is a cofibrant replacement in the category of commutative $R$-algebras, then there is a stable equivalence of commutative $R$-algebras $B^{\cy}(T^{\cI}(\alpha)^{\cof})\xr{\sim}T^{\cI}(B^{\cy}(\alpha))$.
\end{enumerate}
\end{proposition}
Notice, that in (i) the conditions on $M$ hold if $M$ is cofibrant as an object in $\cA\Top^{\cI}$, and in (ii) the conditions on $M$ hold if $M$ is cofibrant as an object in $\cC\Top^{\cI}$. 

\begin{proof}
Consider first case (i) of the proposition where we have the maps of simplicial $R$-modules
\begin{equation}\label{eq:BcyTcof-equivalence}
B^{\cy}_{\bullet}(T^{\cI}(\alpha)^{\cof})\to B^{\cy}_{\bullet}(T^{\cI}(\alpha))\to T^{\cI}(B^{\cy}_{\bullet}(\alpha))
\end{equation}
induced by the cofibrant replacement and the lax symmetric monoidal structure of $T^{\cI}$. Since the underlying $R$-module of $T^{\cI}(\alpha)^{\cof}$ is flat, it follows from Proposition~\ref{prop:TI-derived-monoidal} that the composite map is a stable equivalence in each simplicial degree. We shall argue below that the geometric realization is also a stable equivalence. Composing with the isomorphism $|T^{\cI}(B^{\cy}_{\bullet}(\alpha))|\cong T^{\cI}(B^{\cy}(\alpha))$ from Lemma~\ref{lem:Thom-realization}, we then get the stable equivalence in the proposition. 

The reason why an extra argument is needed to ensure that the geometric realization is a stable equivalence is that the simplicial spectra in question are not necessarily ``good'', that is, the degeneracy maps may fail to be $h$-cofibrations. For this reason we introduce an auxiliary $R$-algebra $T^{\cI}(\alpha)^c$ by choosing a cofibrant replacement $U(\alpha)^{\cof}\to U(\alpha)$ of the $G$-algebra $U(\alpha)$ (see Proposition~\ref{prop:lifted-model-str-on-SID}) and setting 
\[
T^{\cI}(\alpha)^c=B(\bS^{\cI}[U(\alpha)^{\cof}],\bS^{\cI}[G],R).
\]
Then we have a stable equivalence $T^{\cI}(\alpha)^c\to T^{\cI}(\alpha)$ and we claim that the composite map 
\begin{equation}\label{eq:BcyT^c-equivalence} 
B^{\cy}(T^{\cI}(\alpha)^c)\to B^{\cy}(T^{\cI}(\alpha))\to T^{\cI}(B^{\cy}(\alpha))
\end{equation}
is a stable equivalence. This composition admits a factorization
\[
B^{\cy}(T^{\cI}(\alpha)^c) \to B(\bS^{\cI}[B^{\cy}(U(\alpha)^{\cof})],\bS^{\cI}[G],R)\to 
T^{\cI}(B^{\cy}(\alpha))
\]  
and we shall prove that both of these maps are stable equivalences. 

The first map is the geometric realization of a map of simplicial spectra induced by the lax symmetric monoidal structure of the functor $B(\bS^{\cI}[-],\bS^{\cI}[G],R)$. Since the underlying $G$-module of $U(\alpha)^{\cof}$ is flat by Lemma~\ref{lem:gen-M-algebras-underlying-flat}, it follows from 
Lemma~\ref{lem:bar-monoidal-equivalence} that the underlying simplicial map is a levelwise stable equivalence. Furthermore, having replaced $U(\alpha)$ by the cofibrant (hence well-based) $\cI$-space monoid $U(\alpha)^{\cof}$ we ensure that the degeneracy maps of these simplicial spectra are $h$-cofibrations. Hence the geometric realization is also a stable equivalence.
The second map is induced by the composition 
\begin{equation}\label{eq:BcyU-UBcy-equivalence}
B^{\cy}(U(\alpha)^{\cof})\to B^{\cy}(U(\alpha)) \to U(B^{\cy}(\alpha))
\end{equation}
given by the cofibrant replacement $U(\alpha)^{\cof}\to U(\alpha)$ and the lax symmetric monoidal structure of $U$ from Lemma~\ref{lem:ISG-ISBG-adj-monoidal}. By the homotopy invariance of the functor $B(\bS^{\cI}[-],\bS^{\cI}[G],R)$ it suffices to show that the composition in \eqref{eq:BcyU-UBcy-equivalence} is indeed an $\cI$-equivalence. Using again that the underlying $G$-module of $U(\alpha)^{\cof}$ is flat, an argument based on the cellular filtration shows that also the underlying $G$-module of $B^{\cy}(U(\alpha)^{\cof})$ is flat. Thus, we are in a position to use that the $(V,U)$-adjunction in Theorem~\ref{thm:Top^I_G-Quillen-equivalence} is a Quillen equivalence so that showing the map in \eqref{eq:BcyU-UBcy-equivalence} to be an $\cI$-equivalence is equivalent to showing that the adjoint map
\[
V(B^{\cy}(U(\alpha)^{\cof}))\cong B^{\cy}(V(U(\alpha)^{\cof})) \to B^{\cy}(\alpha)
\]
is an $\cI$-equivalence in $\Top^{\cI}/BG$. The latter map is induced by the derived counit of the adjunction  $V(U(\alpha)^{\cof})\to \alpha$ and is therefore an $\cI$-equivalence, cf.\ \cite[Proposition~1.3.13]{Hovey_model}. This concludes the argument why the map in \eqref{eq:BcyT^c-equivalence} is a stable equivalence.

Now we can finally prove that the geometric realization of \eqref{eq:BcyTcof-equivalence} is a stable equivalence. We may assume without loss of generality that the cofibrant replacement $T^{\cI}(\alpha)^{\cof}\to T^{\cI}(\alpha)$ is an acyclic fibration. Let $(T^{\cI}(\alpha)^c)^{\cof}\to T^{\cI}(\alpha)^c$ be a cofibrant replacement of the $R$-algebra $T^{\cI}(\alpha)^c$. Then we can lift the composition  
$(T^{\cI}(\alpha)^c)^{\cof}\to T^{\cI}(\alpha)^c\to T^{\cI}(\alpha)$ to a stable equivalence $(T^{\cI}(\alpha)^c)^{\cof}\to T^{\cI}(\alpha)^{\cof}$. Passing to the cyclic bar constructions, we get a diagram of stable equivalences 
\[
B^{\cy}(T^{\cI}(\alpha)^c) \xl{\sim} B^{\cy}((T^{\cI}(\alpha)^c)^{\cof}) \xr{\sim} B^{\cy}(T^{\cI}(\alpha)^{\cof})
\]
over $B^{\cy}(T^{\cI}(\alpha))$. Here we use that the underlying $R$-module of $T^{\cI}(\alpha)^c$ is flat by Lemma~\ref{lem:invariant-bar}. Having proved that the map in \eqref{eq:BcyT^c-equivalence} is a stable equivalence, it follows that the geometric realization of \eqref{eq:BcyTcof-equivalence} is a stable equivalence which gives the statement in (i). 

The proof of case (ii) proceeds as above except that we now choose $U(\alpha)^{\cof}$ to be a cofibrant replacement in the model structure on commutative $G$-algebras over $EG$ (see Proposition~\ref{prop:lifted-model-str-on-SID}). 
\end{proof}

Now we come to our main result in this section which we shall formulate in terms of Thom spectra associated to space level data over $BG_{h\cI}$. Since the commutative $\cI$-space monoid $BG$ is grouplike ($BG_{h\cI}$ is path connected), it follows that $BG_{h\cI}$ inherits the structure of an infinite loop space. Let us write $B^n(BG_{h\cI})$ for the deloopings as defined in  \cite[Section~5.2]{Schlichtkrull_units}. The first delooping $B^1(BG_{h\cI})$ is canonically equivalent to the usual bar construction $B(BG_{h\cI})$ by \cite[Proposition~5.3]{Schlichtkrull_units}.
We say that a map of based spaces $f\colon K\to BG_{h\cI}$ is an $n$-fold loop map if there exists an $n$-fold delooping $B^nK$ of $K$ and a based map $B^nf\colon B^nK\to B^n(BG_{h\cI})$ such that the diagram
\[
\xymatrix@-1pc{
\Omega^n(B^nK) \ar[rr]^-{\Omega^n(B^n f)} && \Omega^n(B^n(BG_{h\cI}))\\
K \ar[rr]^-f \ar[u]^{\simeq}  && BG_{h\cI} \ar[u]_{\simeq}.
}
\]
is homotopy commutative. 
Using the machinery detailed in Appendix~\ref{app:loop-rectification}, every one-fold loop map can be realized up to weak homotopy equivalence as a map of grouplike topological monoids. Since we are mainly interested in strictly associative $R$-algebra spectra, it will be most convenient to state our results for Thom spectra associated to grouplike monoids over $BG_{h\cI}$. 

As preparation we recall a construction
from~\cite[Section~1]{Blumberg-C-S_THH-Thom}. Consider the unstable Hopf map  $\eta \colon S^3 \to S^2$ and the homotopy class of maps defined by the composition
\[
\eta \colon B(BG_{h\cI}) \simeq \mathrm{Map}_*(S^2, B^3 (BG_{h\cI})) \xrightarrow{\eta^*} \mathrm{Map}_*(S^3,B^3(BG_{h\cI})) \simeq BG_{h\cI}.
\]
Let $L(-)$ denote the free loop space functor. Applied to the infinite loop space $B(BG_{h\cI})$, the canonical splitting of the fibration sequence
\[
\Omega(B(BG_{h\cI}))\to L(B(BG_{h\cI})) \to B(BG_{h\cI})
\]
and the equivalence $BG_{h\cI}\xr{\sim} \Omega(B(BG_{h\cI}))$ gives a canonical product decomposition $L(B(BG_{h\cI}))\simeq BG_{h\cI}\times B(BG_{h\cI})$.
\begin{definition}\label{def:L-eta-map}
Given a well-based and grouplike topological monoid $M$ and a map of topological monoids $f\colon M \to BG_{h\cI}$, we use the notation $L^{\eta}(Bf)$ for a representative of the homotopy class defined by the composition
\[
L(BM) \xrightarrow{L(B(f))} L(B(BG_{h\cI})) \simeq BG_{h\cI} \times B(BG_{h\cI}) \xrightarrow{\id \times \eta} BG_{h\cI} \times BG_{h\cI} \to BG_{h\cI}
\]
where the last map is the multiplication of $BG_{h\cI}$.
\end{definition}
 
\begin{theorem}\label{thm:thh-thom} Consider the $R$-algebra Thom spectrum $T(f)$ associated to a map of topological monoids $f\colon M\to BG_{h\cI}$ where $M$ is well-based and grouplike. 
\begin{enumerate}[(i)]
\item
The $R$-module $\THH^R(T(f))$ is stably equivalent to  $T(L^{\eta}(Bf))$.
\item 
If $f$ is a $2$-fold loop map, then the $R$-module $\THH^R(T(f))$ is stably equivalent to $T(f)^{\cof}\sm_RT(\eta \circ Bf)$, where  $T(f)^{\cof}$ is a cofibrant replacement of $T(f)$ as an $R$-module and $\eta \circ Bf\colon BM\to B(BG_{h\cI})\to BG_{h\cI}$ is defined as above.
\item 
If $f$ is a $3$-fold loop map, then the $R$-module $\THH^R(T(f))$ is stably equivalent to $T(f)\sm BM_+$.  
\end{enumerate}
\end{theorem}

\begin{proof}
Let $T(f)^{\cof}\to T(f)$ be a cofibrant replacement of the $R$-algebra $T(f)$ and let us take $B^{\cy}(T(f)^{\cof}))$ as our model of $\THH^R(T(f))$, cf.\ Remark~\ref{rem:THH-cof-replacement}. By definition, $T(f)=T^{\cI}(P_{BG}(f))$ where $P_{BG}(f)$ is the monoid in $\Top^{\cI}/BG$ obtained by applying the $\cI$-spacification functor to $f$. Let $P_{BG}(f)^{\cof}\to P_{BG}(f)$ be a cofibrant replacement as an $\cI$-space monoid over $BG$. Then the underlying $\cI$-space of $P_{BG}(f)^{\cof}$ is flat so that we have a chain of stable equivalences
\[
B^{\cy}(T^{\cI}(P_{BG}(f))^{\cof}) \xl{\sim} B^{\cy}(T^{\cI}(P_{BG}(f)^{\cof})^{\cof})\xr{\sim} T^{\cI}(B^{\cy}(P_{BG}(f)^{\cof}))
\]
where the first equivalence is induced by the cofibrant replacement and the second is given by Proposition~\ref{prop:Bcy-vs-T}.
Furthermore, it follows from Proposition~\ref{prop:TI-alpha-vs-TalphahI} that the last term is stably equivalent to 
$T(B^{\cy}(P_{BG}(f)^{\cof})_{h\cI})$. The rest of the proof follows the outline in \cite{Blumberg-C-S_THH-Thom}: For (i) we first use the argument from \cite[Proposition~4.8]{Blumberg-C-S_THH-Thom} to show that the domain of $B^{\cy}(P_{BG}(f)^{\cof})_{h\cI}$ is weakly equivalent to $L(BM)$. The proof of \cite[Theorem 2.2]{Blumberg-C-S_THH-Thom} then shows that under this equivalence, $B^{\cy}(P_{BG}(f)^{\cof})_{h\cI}$ represents the homotopy class of $L^{\eta}(Bf)$.
(As explained in~\cite[Section 8.1]{Blumberg-C-S_THH-Thom}, the functor $U$ in the proof of that theorem is $(-)_{h\cI}$ if one works in $\cI$-spaces.) Since our Thom spectrum functor sends products to derived smash products (Proposition~\ref{prop:T-derived-monoidal}) and preserves tensors with spaces (Proposition~\ref{prop:properties-of-T}),  the proofs of ~\cite[Theorem 2 and 3]{Blumberg-C-S_THH-Thom} in~\cite[Section 3.3]{Blumberg-C-S_THH-Thom} apply almost verbatim to give (ii) and~(iii).  
\end{proof}

\section{Modules and classifying spaces for commutative \texorpdfstring{$\cI$}{I}-space monoids} \label{sec:modules-classifying-spaces}
In this section $G$ denotes a commutative $\cI$-space monoid and we shall continue the analysis of the module category $\Top^{\cI}_G$ initiated in Section~\ref{subsec:class-spaces}. The primary aim is to finish the proof of 
Theorem~\ref{thm:Top^I_G-Quillen-equivalence} stating  that $\Top^{\cI}_G$ is Quillen equivalent to $\Top^{\cI}/BG$ provided that $G$ is grouplike and cofibrant. Here we recall that both the absolute and the positive flat model structure on $\Top^{\cI}$ lift to corresponding absolute and positive flat model structures on $\Top^{\cI}_G$. We shall use the term \emph{flat $G$-module} for a cofibrant object in the absolute flat model structure on $\Top^{\cI}_G$.

\begin{lemma}\label{lem:boxtimes_M-inv} 
If $X$ is a flat $G$-module, then the endofunctor $X\boxtimes_G(-)$ on $\Top^{\cI}_G$ preserves $\cI$-equivalences. 
\end{lemma}
\begin{proof}
  For the proof we may assume without loss of generality that $X$ is
  the colimit of a $\lambda$-sequence of $G$-modules
  $\{X_{\alpha}\colon\alpha<\lambda\}$ (for some ordinal $\lambda$)
  such that $X_0$ is the initial $G$-module and the map $X_{\alpha}\to
  X_{\alpha+1}$ is obtained by cobase change from a map of the form
  $K_{\alpha}\boxtimes G\to L_{\alpha}\boxtimes G$, where
  $K_{\alpha}\to L_{\alpha}$ is a generating cofibration for the absolute flat
  model structure on $\Top^{\cI}$. Given a $G$-module $Y$,
  $X\boxtimes_GY$ is then the colimit of the $\lambda$-sequence
  $\{X_{\alpha}\boxtimes_G Y\colon \alpha<\lambda\}$, where
  $X_{\alpha}\boxtimes_GY\to X_{\alpha+1}\boxtimes_GY$ is the cobase
  change of $K_{\alpha}\boxtimes Y\to L_{\alpha}\boxtimes Y$.
  By~\cite[Proposition 7.1(ii) and (vi)]{Sagave-S_diagram} this is a
  $\lambda$-sequence of $h$-cofibrations in the sense
  of~\cite[Section~7]{Sagave-S_diagram}.  Now let $Y\to Y'$ be an
  $\cI$-equivalence between $G$-modules. Using the gluing lemma
  for $h$-cofibrations and $\cI$-equivalences, \cite[Proposition
  7.1(iv)]{Sagave-S_diagram}, we argue by induction to see that
  $X_{\alpha}\boxtimes_GY\to X_{\alpha}\boxtimes_GY'$ is an
  $\cI$-equivalence for all $\alpha$. By
  \cite[Proposition~7.1(v)]{Sagave-S_diagram} the map of colimits is
  therefore also an $\cI$-equivalence.
\end{proof}

\begin{lemma}\label{lem:bar-substitution}
If~$X$ is a flat $G$-module, then the map $B(X,G,Y)\to X\boxtimes_G Y$ is an $\cI$-equivalence. 
\end{lemma}
\begin{proof}
  Using the canonical isomorphism $B(X,G,Y)\cong
  X\boxtimes_G B(G,G,Y)$, it suffices by Lemma~\ref{lem:boxtimes_M-inv} to show that the canonical map
  $B(G,G,Y)\to Y$ is an $\cI$-equivalence. However, this map is even a
  level equivalence since $Y$ is a simplicial deformation retract of
  the domain before passing to the geometric realization.
\end{proof}

Recall that we use the term \emph{cofibrant commutative $\cI$-space monoid} to mean a cofibrant object in the positive flat model structure on $\cC\Top^{\cI}$.

\begin{lemma}\label{lem:boxtimes-G-square}
   Let $G$ be a grouplike and cofibrant commutative $\cI$-space monoid, let $X\to X'$
and $Y\to Y'$ be maps of $G$-modules, and suppose that the $G$-modules $X$ and $X'$ are flat. Then the
  commutative square
  \[
  \xymatrix@-1pc{
    X\boxtimes_G Y \ar[r]\ar[d] & X\boxtimes_G Y'\ar[d]\\
    X'\boxtimes_G Y\ar[r] & X'\boxtimes_G Y' }
  \]
is homotopy cartesian. 
\end{lemma}
\begin{proof}
Using Lemma~\ref{lem:bar-substitution} and that the homotopy colimit functor $(-)_{h\cI}$ detects homotopy cartesian squares 
\cite[Corollary 11.4]{Sagave-S_diagram}, it suffices to show that the diagram of spaces
\[
\xymatrix@-1pc{
B(X,G,Y)_{h\cI} \ar[r] \ar[d]& B(X,G,Y')_{h\cI} \ar[d]\\
B(X',G,Y)_{h\cI} \ar[r] & B(X',G,Y')_{h\cI}
}
\]
is homotopy cartesian. It follows from \cite[Proposition 3.15(ii)]{Sagave-S_diagram} that the underlying $\cI$-space of $G$ is flat and inspecting the generating  cofibrations of the absolute flat model structure on $\Top^{\cI}_G$, we see that the underlying $\cI$-spaces of $X$ and $X'$ are also flat. The topological monoid $G_{h\cI}$ is well-based by \cite[Proposition 12.7]{Sagave-S_diagram}. Since the monoidal structure map of $(-)_{h\cI}$ in \eqref{eq:monoidal-structure-map-hI} is a weak equivalence if one of the factors is flat  \cite[Lemma 2.25]{Sagave-S_group-compl} and the relevant simplicial objects are good in the usual sense, this implies that the above diagram is weakly equivalent to the left hand square in the diagram 
\[
  \xymatrix@-1pc{ B(X_{h\cI},G_{h\cI},Y_{h\cI}) \ar[r] \ar[d]&
    B(X_{h\cI},G_{h\cI},Y'_{h\cI})
    \ar[r]\ar[d] & B(X_{h\cI}, G_{h\cI}, *)\ar[d]\\
    B(X'_{h\cI},G_{h\cI},Y_{h\cI}) \ar[r] &
    B(X'_{h\cI},G_{h\cI},Y'_{h\cI}) \ar[r] & B(X'_{h\cI},G_{h\cI},*).  }
  \]
Here the horizontal maps in the right hand square are induced by the projection $Y'_{h\cI}\to *$. By
  \cite[Theorem~7.6]{May-classifying}, the assumption that $G_{h\cI}$ be grouplike
  implies that the horizontal maps in the right hand square and the
  outer square are quasifibrations. Since these are actual pullback
  diagrams this translates into the statement that these squares are
  homotopy cartesian. The left hand square is therefore also homotopy
  cartesian as claimed.
\end{proof}

\begin{proposition}\label{prop:lifted-model-str-on-SID}
Let $G$ be a commutative $\cI$-space monoid and let $\cD$ be an operad in $\Top$. Then the category $\Top^{\cI}_G[\cD]$ of $\cD$-algebras in $\Top^{\cI}_G$ admits a positive flat model structure where a map is a weak equivalence or fibration if the underlying map in $\Top^{\cI}_G$ (or $\Top^{\cI}$) is so with respect to the positive flat model structure. 
\end{proposition}
\begin{proof}
Using an obvious generalization of the double cell filtration for $\cD$-algebra cell attachments in \cite[Proposition 10.1]{Sagave-S_diagram} to $G$-modules, this is completely analogous to the proof of~\cite[Proposition 9.3]{Sagave-S_diagram}.
\end{proof}

In the next lemma we refer to the \emph{fine model structure} on $\Sigma_k$-spaces discussed in Section~\ref{subsec:I-spaces}.

\begin{lemma}\label{lem:gen-M-algebras-underlying-flat}
Let $G$ be a commutative $\cI$-space monoid and let $\cD$ be an operad such that each space $\cD(k)$ is cofibrant as an object in the category of $\Sigma_k$-spaces equipped with the fine model structure. If $A$ is an object in $\Top^{\cI}_{G}[\cD]$ that is cofibrant in the positive flat model structure, then the underlying $G$-module of $A$ is flat.
\end{lemma}
\begin{proof}
  This is analogous to~\cite[Proposition 12.5]{Sagave-S_diagram}
  (which provides the statement for $G=U^{\cI}$).
\end{proof}

In the case of the commutativity operad $\cC$, the category $\Top^{\cI}_G[\cC]$ can be identified with the category 
$G/\cC \Top^{\cI}$ of commutative $G$-algebras. Under this identification, the lifted model structure on $\Top^{\cI}_G[\cC]$ resulting from Proposition~\ref{prop:lifted-model-str-on-SID} becomes the standard under-category model structure on  $G/\cC \Top^{\cI}$ inherited from the positive flat model structure on $\cC \Top^{\cI}$. With this in mind, the following result may be viewed as a strengthening of Lemma~\ref{lem:gen-M-algebras-underlying-flat} in the case of the commutativity operad. 

\begin{lemma}\label{lem:com-M-algebras-underlying-flat}
Let $G$ be a  commutative $\cI$-space monoid and let $A \to A'$ be a cofibration in the positive flat model structure on $G/\cC \Top^{\cI}$. If the underlying $G$-module of $A$ is flat, then the underlying $G$-module of $A'$ is also flat. 
\end{lemma}
\begin{proof}
By a cell induction argument, the claim reduces to the case where $A'$ is obtained from $A$ by attaching a generating cofibration for the positive flat model structure on $\Top^{\cI}_G[\cC]$. To analyze this pushout, we again use the $G$-module version of the double cell filtration provided by~\cite[Proposition 10.1]{Sagave-S_diagram}. The $\Sigma_k$-action on the object $U^{\mathbb D}_k(A)$ appearing in this proposition is trivial since we consider the commutativity operad here, see~\cite[Example 10.2]{Sagave-S_diagram}.  Hence the claim follows from the $G$-module version of~\cite[Lemma 12.16]{Sagave-S_diagram}.
\end{proof}

This lemma applies in particular to the commutative $G$-algebra $EG$ introduced in Definition~\ref{def:EG-BG}.
\begin{corollary}
If $G$ is a cofibrant commutative $\cI$-space monoid, then the underlying $G$-module of $EG$ is flat.
\end{corollary}
\begin{proof}
Inspecting the skeletal filtration of the commutative $G$-algebra $B(U^{\cI},G,G)$, we conclude that the underlying $G$-module is flat. Hence the claim follows from the previous lemma.  
\end{proof}

Now we can finally complete the proof of Theorem~~\ref{thm:Top^I_G-Quillen-equivalence}
by showing that the $(V,U)$-adjunction in Definition \ref{def:VU-adjunction} defines a Quillen equivalence. We also prove a variant of this result where we allow for actions by an operad in $\Top$.

\begin{proposition}\label{prop:ISG-ISBG-adj}
  Let $G$ be a grouplike and cofibrant commutative $\cI$-space monoid.
  \begin{enumerate}[(i)]  
\item The adjunction $V\colon \Top^{\cI}_G/EG \rightleftarrows
  \Top^{\cI}/BG \colon U$ is a Quillen equivalence with respect to the
  absolute and positive flat model structures.
\item If $\cD$ is an operad that satisfies the condition in Lemma~\ref{lem:gen-M-algebras-underlying-flat}, then the induced adjunction of $\cD$-algebras  $V\colon \Top^{\cI}_G[\cD]/EG \rightleftarrows
  \Top^{\cI}[\cD]/BG \colon U$ is a Quillen equivalence with respect to the positive flat model structures.
\end{enumerate}
\end{proposition}
\begin{proof}
 For (i) we first observe that the $(V,U)$-adjunction is composed of two Quillen adjunctions, hence is itself a Quillen adjunction. 
 Now suppose we are given a cofibrant object $X \to EG$ in $\Top^{\cI}_G/EG$, a fibrant object $Y \to BG$ in $\Top^{\cI}/BG$, and a map of $\cI$-spaces $\varphi\colon X\boxtimes_GU^{\cI}\to Y$ over $BG$. Then we must show that $\varphi$ is an $\cI$-equivalence if and only if its adjoint is. Consider the commutative diagram of $\cI$-spaces
 \[
 \xymatrix@-1pc{
 X \ar[r]^-{\cong} \ar[d]& X\boxtimes_GG \ar[r] \ar[d] &X\boxtimes_GU^{\cI} \ar[r]^-{\varphi} \ar[d]& Y \ar[d]\\
 EG \ar[r]^-{\cong} & EG\boxtimes_GG \ar[r]  &EG\boxtimes_GU^{\cI} \ar[r] & BG 
 }
 \]
 where the middle square is induced by the given map $X\to EG$ and the projection $G\to U^{\cI}$. The latter square is homotopy cartesian by Lemma~\ref{lem:boxtimes-G-square}. Notice that the map $EG\boxtimes_GU^{\cI} \to BG$ in the diagram is an $\cI$-equivalence since the composition
 \[
 B(U^{\cI},G,G)\boxtimes_GU^{\cI}\to EG\boxtimes_GU^{\cI} \to BG
 \]
 is an isomorphism and the first map is an $\cI$-equivalence (the left Quillen functor $(-)\boxtimes_GU^{\cI}$ preserves acyclic cofibrations). 
 Hence $\varphi$ is an $\cI$-equivalence if and only if the right hand square is homotopy cartesian. Furthermore, it follows from the definitions that the adjoint of $\varphi$ is an $\cI$-equivalence if and only if the outer square is homotopy cartesian. Now it is clear that if the right hand square is homotopy cartesian, then the outer square is also homotopy cartesian. In order to show the converse, we use that the functor 
 $(-)_{h\cI}$ detects and preserves homotopy cartesian squares by \cite[Corollary 11.4]{Sagave-S_diagram}. Passing to the associated diagram of homotopy colimits, the bottom horizontal map in the middle becomes surjective so that the result follows by comparing vertical homotopy fibers.

The Quillen adjunction part of (ii) follows from (i). For the
  Quillen equivalence part, we use that if $X\to EG$ is an cofibrant object in 
 $\Top^{\cI}_G[\cD]/EG$, then the underlying $G$-module of $X$ is flat by
  Lemma~\ref{lem:gen-M-algebras-underlying-flat}. Hence the argument
  for (i) applies.
\end{proof}   
  
 \appendix
\section{Rectification of loop maps} \label{app:loop-rectification}
\label{sect:loop-rectification}
Let $\cA\Top$ denote the category of topological monoids and let us for the rest of this section fix a topological monoid $M$ which we assume to be grouplike and well-based. We shall then define a ``loop functor'' with values in the corresponding over-category $\cA\Top/M$. Our main application of this construction  in the paper is for $M=BG_{h\cI}$ (in the notation of Section~\ref{subsec:space-level-T}) in which case it allows us to pass from loop space data to strictly associative $R$-algebra Thom spectra, cf.\ Proposition~\ref{prop:T-operad-version}.

\subsection{The classifying space \texorpdfstring{$B_1M$}{B1 M}}
In order to work within the setting of topological monoids (as opposed to a more relaxed notion of $A_{\infty}$ spaces), it is most elegant to apply the model $B_1M$ of the classifying space functor introduced by Fiedorowicz~\cite{Fiedorowicz_classifying} (which in turn is a variant of May's classifying space functor \cite{May-geometry}). We begin by reviewing the relevant details.  Recall that the Moore loop space $\Lambda(K)$ of a based space $K$ is the subspace of $\Map([0,\infty),K)\times [0,\infty)$ given by the pairs $(f,r)$ such that $f(t)=*$ (the base point) if $t=0$ or $t\geq r$.  The Moore loop space defines a functor from based spaces to topological monoids and the canonical inclusion of the standard loop space $\Omega (K)\to \Lambda (K)$ is an (unbased) deformation retract. We can refine the target category for $\Lambda$ by letting $\Top_{[0,\infty)}$ be the full subcategory of the over-category 
$\Top_*/[0,\infty)$ with objects $p\colon K\to [0,\infty)$ such that $p^{-1}(0)=\{*\}$. (This is the category denoted $\mathcal T_*[\mathbb R_+]$ in \cite{Fiedorowicz_classifying}). The structure map $\Lambda(K)\to[0,\infty)$ is the obvious projection.
Notice that $\Top_{[0,\infty)}$ inherits a symmetric monoidal structure from 
$\Top_*/[0,\infty)$ when we equip $[0,\infty)$ with the additive monoid structure. The purpose of introducing the category $\Top_{[0,\infty)}$ is to realize $\Lambda$ as a right adjoint in an adjunction
\[
\Xi\colon\Top_{[0,\infty)}\rightleftarrows \Top_*:\!\Lambda
\]
in which the left adjoint is the Moore suspension functor defined by
\[
\Xi(K,p)=K\times[0,\infty) /  \{(x,s)|\text{ $s=0$ or $s\geq p(x)$} \}.
\]
Because $(\Xi,\Lambda)$ form an adjoint functor pair, the composition $\Lambda\Xi$ defines a monad on $\Top_{[0,\infty)}$. We also have the symmetric monoidal adjunction 
\[
L\colon \Top_{[0,\infty)} \rightleftarrows \Top_* :\! R
\]
where $L$ is the forgetful functor and $R$ takes a based space $K$ to
\[
RK=\{(x,s)\in K\times [0,\infty)|\text{ $s>0$ or $s=0$ and $x=*$}\} 
\]
with structure map the projection onto $[0,\infty)$. It follows from \cite[Lemma~6.6]{Fiedorowicz_classifying} that there are natural homeomorphisms $R\Omega=\Lambda$ and $\Sigma L=\Xi$. This is the reason for working with $\Top_{[0,\infty)}$ as opposed to $\Top_*/[0,\infty)$. It is easy to check that $R$ takes NDR pairs to NDR pairs (see \cite{Steenrod_convenient}). This in turn implies that if $K$ is a well-based space, then also $RK$ is well-based. 


Now let $J$ denote the classical James construction that to a based space $K$ associates the free topological monoid $J(K)$. 
We may also view $J$ as a monad on $\Top_{[0,\infty)}$ by assigning to an object 
$p\colon K\to [0,\infty)$ the induced map of topological monoids $J(K)\to [0,\infty)$. Defined in this manner, $J(K,p)$ is the free monoid on $(K,p)$ which implies that we have a canonical map of monads $\lambda\colon J\to \Lambda\Xi$ on $\Top_{[0,\infty)}$. It is proved in \cite[Theorem~6.8]{Fiedorowicz_classifying} that this is a weak equivalence when applied to spaces that are well-based and path connected. 

The definition of Fiedorowicz's classifying space $B_1M$ is based on the monadic bar construction and we refer to \cite[Section~9]{May-geometry} for details. First we apply the functor $R$ to get a topological monoid $RM$ that is an algebra for the monad $J$ on 
$\Top_{[0,\infty)}$. Secondly, the adjoint of $\lambda\colon J\to \Lambda\Xi$  makes $\Xi$ a $J$-functor in the sense of \cite{May-geometry}. Putting all this together, $B_1M$ is defined to be the corresponding monadic bar construction 
\[
B_1M=B(\Xi,J,RM). 
\]
The fact that $RM$ is well-based  ensures that this is the realization of a good simplicial space in the sense that the degeneracy maps are $h$-cofibrations (it is clear that $J$ preserves $h$-cofibrations). 

We shall also consider the monadic bar construction $B(J,J,RM)$ which comes with a canonical weak equivalence of topological monoids 
\begin{equation}\label{eq:B(J,J,RM)-M}
B(J,J,RM)\xr{\sim} RM\xr{\sim} M.
\end{equation}
Let us use the notation $\lambda$ for the composite map of topological monoids 
\begin{equation}\label{eq:lamba-equivalence}
\lambda\colon B(J,J,RM)\xr{\sim} B(\Lambda\Xi, J,RM)\xr{\sim} \Lambda (B(\Xi,J,RM))=\Lambda (B_1M)
\end{equation}
where the first map is induced by the natural transformation with the same name and the second map is the canonical weak equivalence relating the geometric realization of the levelwise Moore loop space to the Moore loop space of the geometric realization, cf.\ \cite[Theorem~12.3]{May-geometry}. It remains to show that the first map is a weak equivalence as indicated.

\begin{proposition}\label{prop:lambda-equivalence}
The map $\lambda$ in \eqref{eq:lamba-equivalence} is a weak homotopy equivalence. 
\end{proposition}
\begin{proof}
Notice first that our assumption that $M$ be grouplike implies that also $B(J,J,RM)$ is grouplike. The topological monoid $\Lambda(B_1M)$ is grouplike by definition. Hence it suffices to show that $\lambda$ induces a weak equivalence $B\lambda$ of bar constructions. The latter statement reduces to showing that the first map in \eqref{eq:lamba-equivalence} induces a weak equivalence of bar constructions in every simplicial degree. This essentially follows from \cite[Theorem~6.12]{Fiedorowicz_classifying} which states that $BJ(K,p)\to B\Lambda(\Xi(K,p))$ is a weak equivalence for every well-based object $(K,p)$ in $\Top_{[0,\infty)}$. 
\end{proof}

Combining the weak equivalences in \eqref{eq:B(J,J,RM)-M} and \eqref{eq:lamba-equivalence}, we get a chain of weak equivalences relating the topological monoids $M$ and $\Lambda(B_1M)$. This is one of the advantages of working with the model $B_1M$ of the classifying space.
Applying the ordinary (bar construction) classifying space functor to these equivalences and composing with the weak equivalence 
$B\Lambda(B_1 M)\xr{\sim} B_1M$ from \cite[Lemma~15.4]{May-classifying}, we get the chain of weak homotopy equivalences
\[
BM\xl{\sim} BB(J,J,RM)\xr{\sim} B_1M
\]
as stated in \cite[Theorem~7.3]{Fiedorowicz_classifying}.

\subsection{Rectification of loop maps} \label{subsec:rectified-loop}
Based on the preparations in the previous subsection, we now introduce a rectified loop functor
\[
\Omega'\colon \Top_*/B_1M\to \cA\Top/M;\qquad (g\colon K\to B_1M)\mapsto (\Omega'(g)\colon\Omega'_g(K)\to M)
\]  
that models the looping of a based map $g\colon K\to B_1M$ by a map of actual topological monoids $\Omega'(g)\colon \Omega_g'(K)\to M$. 
In detail, given a based map $g$ as above, we let $\Omega'_g(K)$ be the homotopy pullback of the diagram of topological monoids 
\[
\Lambda(K)\xr{\Lambda(g)} \Lambda(B_1 M) \xl{\lambda} B(J,J,RM)
\] 
and we let $\Omega'(g)$ be the composite map
\[
\Omega'(g)\colon \Omega'_g(K)\to B(J,J,RM)\to M.
\]
It follows from the discussion in the previous subsection that the composition 
\[
\cA\Top/M\xr{B_1} \Top_*/B_1 M\xr{\Omega'} \cA\Top/M
\]
defines a group completion functor on the full subcategory of well-based topological monoids over $M$. Restricted to objects (that is, homomorphisms) $f\colon N\to M$ with $N$ well-based and grouplike, we thus have a chain of natural weak equivalences of topological monoids
$\Omega'_{B_1f}(B_1 N)\simeq N$ over $M$. 

We shall need a further compatibility relation. Let $(B_1(-)\downarrow B_1M)$ be the comma category with objects $(N,g)$ given by a topological monoid $N$ and a based map $g\colon B_1N\to B_1 M$. A morphism $f\colon (N,g)\to (N',g')$ is a homomorphism $f\colon N\to N'$ such that $g=g'\circ Bf$. In the following proposition we view $\Omega'_g(B_1 N)$ and $N$ as functors from $(B_1(-)\downarrow B_1M)$ to 
$\cA\Top$ by composing with the appropriate forgetful functors.
\begin{proposition}\label{prop:Omega'-compatible}
There is a chain of natural maps of topological monoids relating $\Omega'_g(B_1 N)$ and $N$. These maps are weak equivalences when $N$ is well-based and grouplike. 
\end{proposition}
\begin{proof}
It follows from the definition that the homomorphisms
\[
N\xl{\sim} B(J,J,RN)\xr{} \Lambda(B_1 N)\xl{\sim} \Omega'_g(B_1 N)
\] 
satisfy the stated naturality  conditions. The second statement then follows from Proposition~\ref{prop:lambda-equivalence}.
\end{proof}


\end{document}